\documentclass[reqno]{amsart}
\usepackage{amsmath,latexsym,amssymb}
\usepackage[mathscr]{eucal}

\DeclareSymbolFont{bbold}{U}{bbold}{m}{n}
\DeclareSymbolFontAlphabet{\mathbbold}{bbold}

\theoremstyle{plain}
\newtheorem{theorem}{Theorem}[section]
\newtheorem{lemma}[theorem]{Lemma}
\newtheorem{lemma:dual}[theorem]{Duality Lemma}
\newtheorem{lemma:fdual}[theorem]{Full Duality Lemma}
\newtheorem{lemma:transfer}[theorem]{New-from-old Lemma}
\newtheorem{theorem:transfer}[theorem]{New-from-old Theorem}

\theoremstyle{definition}
\newtheorem{definition}[theorem]{Definition}
\newtheorem{notation}[theorem]{Notation}
\newtheorem{example}[theorem]{Example}
\newtheorem{remark}[theorem]{Remark}
\newtheorem{warning}[theorem]{Warning}
\newtheorem{facts}[theorem]{Facts}
\newtheorem{note}[theorem]{Note}
\newtheorem{algorithm}[theorem]{Algorithm}
\newtheorem{assumptions}[theorem]{Assumptions}

\newenvironment{newlist}
     {\begin{list}{}{\setlength{\labelsep}{0.75em}
                     \setlength{\labelwidth}{1.25em}
                     \setlength{\leftmargin}{2em}
                     \setlength{\itemsep}{1ex}}}
     {\end{list}}

\newcommand{\Ahom}{(hm)}
\newcommand{\Aop}{(op)}
\newcommand{\Aaxi}{(ax)}

\newcommand{\up}[1]{\textup{#1}}

\usepackage{pgf,tikz}
\usetikzlibrary{arrows,calc,positioning,intersections}

\tikzset{%
 shaded/.style={draw, shape=circle, fill=black!35, inner sep=1.4pt},
 unshaded/.style={draw, shape=circle, fill=white, inner sep=1.4pt},
 straight/.style={->, semithick, >=latex, shorten >=5pt, shorten <=5pt},
 curvy/.style={->, semithick, >=latex, shorten >=4pt, shorten <=4pt, looseness=1.2, bend angle=40},
 short/.style={->, semithick, >=latex, shorten >=2pt, shorten <=3pt, looseness=1.4, bend angle=50},
 loopy/.style={->, semithick, >=latex, shorten >=4pt, shorten <=4pt, min distance=18pt},
 order/.style={thin},
 label/.style={shape=rectangle, inner sep=0pt},
 auto}

\newcommand{\A}{{\mathbf A}}
\newcommand{\pb}{{\mathbf p}}
\newcommand{\rb}{{\mathbf r}}
\renewcommand{\sb}{{\mathbf s}}
\newcommand{\X}{{\mathbb X}}
\newcommand{\Y}{{\mathbb Y}}
\newcommand{\Z}{{\mathbb Z}}

\newcommand{\MB}{{\mathbf M}}
\newcommand{\QB}{{\mathbf Q}}

\newcommand{\ThB}{{\mathbf 3}}

\newcommand{\MT}{{\mathbb M}}
\newcommand{\QT}{{\mathbb Q}}
\newcommand{\ST}{{\mathbb S}}
\newcommand{\TwT}{{\mathbbold 2}}
\newcommand{\ThT}{{\mathbbold 3}}

\newcommand{\CA}{{\boldsymbol{\mathscr{A}}}}
\newcommand{\CX}{{\boldsymbol{\mathscr{X}}}}
\newcommand{\CY}{{\boldsymbol{\mathscr{Y}}}}
\newcommand{\CZ}{{\boldsymbol{\mathscr{Z}}}}

\newcommand{\clo}[1]{\operatorname{Clo}_{\mathrm{ep}}(#1)}
\newcommand{\cadef}[1]{\operatorname{Rel}_{\mathrm{ca}}(#1)}
\newcommand{\dom}[1]{\operatorname{dom}(#1)}
\newcommand{\graph}[1]{\operatorname{graph}(#1)}

\newcommand{\ISP}[1]{\operatorname{\mathsf{ISP}}(#1)}
\newcommand{\IScP}[1]{\operatorname{\mathsf{IS_cP}}^{\scriptscriptstyle +}\!(#1)}
\newcommand{\ThuH}[1]{\operatorname{Th_{uH}}(#1)}

\renewcommand{\And}{\ {\&}\ }
\newcommand{\bigand}{\mathop{\hbox{\upshape\Large\&}}}

\newcommand{\psigma}[1]{\sigma^{\mathbf{#1}^\sharp}}
\newcommand{\rel}[2]{\mathrm{pr}_{#1}(#2)}
\newcommand{\T}{{\mathscr T}}

\renewcommand{\ge}{\geqslant}
\renewcommand{\le}{\leqslant}
\renewcommand{\phi}{\varphi}

\DeclareMathOperator{\Fix}{Fix}
\newcommand{\rhom}[1]{\mathrm{hom}_{#1}}
\newcommand{\Set}{\textsc{Set}}
\newcommand{\op}{\mathrm{op}}
\newcommand{\id}{\mathrm{id}}

\begin{document}

\title{New-from-old full dualities via axiomatisation}

 \author{Brian A. Davey}
 \address{Department of Mathematics and Statistics\\La Trobe University\\Victoria 3086, Australia}
 \email{B.Davey@latrobe.edu.au}

 \author{Jane G. Pitkethly}
 \address{Department of Mathematics and Statistics\\La Trobe University\\Victoria 3086, Australia}
 \email{J.Pitkethly@latrobe.edu.au}

 \author{Ross Willard}
 \address{Department of Pure Mathematics\\University of Waterloo\\Ontario N2L 3G1, Canada}
 \email{rdwillar@uwaterloo.ca}



 \thanks{The third author was supported by a Discovery Grant from NSERC, Canada.}

\begin{abstract}
We study different full dualities based on the same finite algebra.
Our main theorem gives conditions on two different alter egos of a finite
algebra under which, if one yields a full duality, then the other does too.
We use this theorem to obtain a better understanding of several important
examples from the theory of natural dualities.
We also clarify what it means for two full dualities based on the same
finite algebra to be different.
Throughout the paper, a fundamental role is played by the universal
Horn theory of the dual categories.
\end{abstract}

\maketitle


\section{Introduction}

This paper is a contribution to our understanding of full dualities.
We show how to obtain new full dualities from an existing full duality using
a universal Horn axiomatisation of the dual category. This provides a systematic
technique for finding full dualities that can be used to `rediscover' several
important but previously ad-hoc counterexamples from the theory of natural dualities.

In this introduction, we first motivate the theory of natural dualities within
a setting appropriate to a reader with a background in category theory rather
than universal algebra. We then describe the problems in which we are interested
and summarise our results.

The rest of the paper is structured as follows:
We give a focused introduction to full dualities in
Sections~\ref{sec:fulldualities}--\ref{sec:alteregos}.
In preparation for the proof of our main theorem,
we present a motivating example in Section~\ref{sec:three},
and then sketch the relevant universal Horn logic for signatures
that include partial operations in Section~\ref{sec:logic}.
The main theorem (Theorem~\ref{thm:trans})
is proved over Sections~\ref{sec:sharp}--\ref{sec:trans}.
Corollaries and applications of this theorem are developed in
Sections~\ref{sec:thm}--\ref{sec:end}.

\subsection*{Our setting}

The results presented here arise from the study of
`structural dual equivalences' (known as \emph{natural dualities})
for certain concrete categories over $\Set$. The categories that we wish to dualise
are quasivarieties~$\CA$ of $\Sigma$-algebras (with the usual morphisms),
for some finitary functional signature~$\Sigma$, with the added requirement that
$\CA$ is generated by a finite algebra~$\MB$ in the model-theoretic sense
(meaning that $\MB$ is a cogenerator of~$\CA$ and its underlying set is finite).
The theory of natural dualities has its roots in the early 1980s when Davey and Werner set
down the basic theory~\cite{DW:old}. The state of the theory up to the late 1990s
is presented in the text by Clark and Davey~\cite{CD:book}.

\subsubsection*{The primary goal}

As observed by Isbell~\cite{Isbell}, an adjoint connection between
two categories is normally induced by an `object living in both categories'.
Accordingly, starting from such $\CA$ and $\MB$ as described above,
the theory follows a standard recipe and seeks
to place $\CA$ in a dual adjunction by finding
another concrete category $\CZ$ and an object $\MT$ of~$\CZ$ with the same
underlying set as~$\MB$ such that `$\MT$~commutes with~$\MB$';
see Johnstone~\cite[VI.4]{Johnstone}.
The object $\MT$ is known as an \emph{alter ego} of~$\MB$.

Under the right conditions (see, e.g., Porst and Tholen~\cite[1-C]{PT}),
there exist concrete functors
\[
D\colon \CA\to\CZ^\op \quad\text{and}\quad E\colon \CZ\to\CA^\op
\]
represented by
\[
\rhom\CA(-,\MB)\colon \CA\to\Set^\op \quad\text{and}\quad \rhom\CZ(-,\MT)\colon \CZ\to\Set^\op,
\]
respectively, which are part of a dual adjunction between $\CA$ and~$\CZ$.
Letting $\eta\colon \id_\CA \to ED$ and $\varepsilon\colon \id_\CZ \to DE$
denote the units of the adjunction, we obtain a dual equivalence between the fixed subcategories
\begin{align*}
 \Fix\eta &= \{\, \A \in \CA \mid \eta_\A \text{ is an isomorphism} \,\} \quad\text{and} \\
 \Fix\varepsilon &= \{\, \Z \in \CZ \mid \varepsilon_\Z \text{ is an isomorphism} \,\}
\end{align*}
in the usual way (Lambek and Scott~\cite[Prop.~4.2]{LS}); see Section~\ref{sec:fulldualities}
below for details. The primary goal is to find $\CZ$ and $\MT$ such that $\Fix\eta = \CA$.

\subsubsection*{Our choice of dual categories}

A wealth of examples guide our choice for the concrete categories~$\CZ$.
We list just three:
\begin{itemize}
\item
Stone duality~\cite{S:bool} between Boolean algebras and Boolean spaces
(that is, compact totally disconnected spaces, also known as Stone spaces);
\item
Hofmann--Mislove--Stralka duality~\cite{HMS:semi} between unital semilattices
and Bool\-ean topological unital semilattices;
\item
Priestley duality~\cite{P:dist1} between bounded distributive lattices and
Priestley spaces (that is, compact totally order-disconnected ordered spaces).
\end{itemize}
Motivated by examples such as these, and following Clark and Krauss~\cite{CK:topqv}
and Davey and Werner~\cite{DW:old}, we confine our search to concrete categories
$\CZ$ of the following special kind:
for some finitary signature~$\Delta$, the category $\CZ$ consists of all Boolean
spaces enriched with $\Delta$-structure that is continuous (for
operations) and closed (for relations) with respect to the topology.
This restriction to such categories $\CZ$ is severe, and makes the entire project
impossible for some quasivarieties~$\CA$ (for example, the quasivariety of implication
algebras \cite[pp.~148--151]{DW:old}).

We say that $\CA$ is \emph{dualisable} (in our strict sense) if,
for some (equivalently, for every~\cite{DW:ind,S:dual})
finite cogenerator~$\MB$, there exists a category $\CZ$ of this kind
containing an alter ego $\MT$ of~$\MB$ such that $\Fix\eta = \CA$.
In this setting, the requirement that `$\MT$~commutes with~$\MB$' becomes
`$\MT$~is compatible with~$\MB$'; see Section~\ref{sec:fulldualities} for the formal definition.

The categories $\CZ$ that we consider have
concrete powers and an internal notion of `induced structural subobject',
so that if $\A$ is an object in~$\CA$ with underlying set~$A$,
then $D(\A)$ is the induced substructure of $\MT^A$ with topologically closed
underlying set $\rhom\CA(\A,\MB)$. Hence, each object in $\Fix\varepsilon$ is
isomorphic to a topologically closed induced substructure of a non-zero power of~$\MT$.
We summarise this observation by writing $\Fix\varepsilon \subseteq \IScP\MT$.
Following \cite{CK:topqv}, we call $\IScP\MT$ the \emph{topological quasivariety}
generated by~$\MT$. Again guided by examples like the three listed above,
we identify $\IScP\MT$ as a `structurally simple' subcategory of~$\CZ$.

Being greedy, we aim to find $\CZ$ and $\MT$ that not only
satisfy $\Fix\eta = \CA = \ISP\MB$, but also satisfy $\Fix\varepsilon = \IScP\MT$.
In this case, we deem $\Fix\varepsilon$ to be adequately understood
and say that $\CA$ is \emph{fully dualised} within~$\CZ$ via the pair $(\MB,\MT)$.
Since the algebra $\MB$ determines the quasivariety~$\CA$ and
the alter ego~$\MT$ determines both $\CZ$ and $\IScP\MT$, we usually localise
to the pair $(\MB,\MT)$ and say simply that $\MT$~fully dualises~$\MB$.
(Note that, as for dualisability, the full dualisability of $\CA$ is independent
of the choice of the cogenerator~$\MB$~\cite{DHP:indFD}.)

\subsubsection*{Partial operations cannot be avoided}

A fact worth noting is that the signature of~$\CZ$ often must contain operations,
or even partial operations, if $\CA$ is to be fully dualised within $\CZ$ via
a pair $(\MB,\MT)$. For example, every endomorphism of $\MB$ must be represented as
a term function of~$\MT$ (Davey, Haviar and Willard~\cite[Prop.~4.3(b)]{DHW:?}).
The proof is easy, so we sketch it here:
Let~$\X$ be the smallest induced substructure of $D(\MB)$ containing~$\id_M$.
Then $\X \in \IScP\MT$ and the underlying set of~$\X$ consists of all total unary term
functions of~$\MT$. It can be shown that $u\colon E(\X) \to \MB$ defined by
$u(h)=h(\id_M)$ is a bijective morphism in~$\CA$, so is an isomorphism.
Assuming $\CA$ is fully dualised within~$\CZ$ via $(\MB,\MT)$,
we have $\X \in \Fix\varepsilon$ and so $\X \cong DE(\X) \cong D(\MB)$.
But $\X$ is an induced substructure of~$D(\MB)$, so by finiteness, $\X=D(\MB)$.
Since the underlying set of $D(\MB)$ is the set of endomorphisms of~$\MB$,
the result follows.

More generally, there exist quasivariety--cogenerator pairs $(\CA,\MB)$
that are fully dualisable, but only via $(\CZ,\MT)$ whose signature contains
partial operations~\cite{DHW:?}. So we allow the signature of~$\CZ$ to include partial operations.

\subsubsection*{Structural embeddings}

The decision to allow the signature of~$\CZ$ to include partial operations comes
at a price: the internal notion of `induced structural subobject',
as well as the corresponding notion of `structural embedding', becomes fragile.
In the absence of partial operations, structural embeddings in~$\CZ$ are
injective morphisms that reflect the relations in the signature;
they are characterised categorically in~$\CZ$ as concrete embeddings
(monomorphisms that are initial with respect to the un\-derlying-set functor),
and also as regular monomorphisms (equalisers of morphism pairs).
However, once partial operations are allowed in the signature of~$\CZ$,
the two categorical notions split; concrete embeddings in the new setting are
injective morphisms that reflect relations and graphs of partial operations,
while regular monomorphisms must also reflect the domains of partial operations.
Since each inclusion $D(\A)\hookrightarrow \MT^A$ is an embedding in the
latter (stronger) sense, we adopt the latter sense as the `correct' notion
of structural embedding, and take $\IScP\MT$ to mean the full subcategory
of~$\CZ$ consisting of the objects that structurally embed (in the
stronger sense) into a non-zero power of~$\MT$.

\subsection*{Problems of interest}

With our setting described, we turn to the problems that interest us.
Primarily, we would like to know which quasivarieties $\CA$ with a finite
cogenerator $\MB$ are fully dualisable in our sense, and why.
We do not address this problem here.  Nevertheless, our aim here is closely related:
\begin{itemize}
\item
For fixed $\CA$ and $\MB$ with $\CA$ fully dualisable,
we seek to understand \emph{all} $\CZ$ and $\MT$ for which $\CA$
is fully dualised within $\CZ$ via $(\MB,\MT)$.
\end{itemize}
To aid the discussion that follows, let $H_\Omega$ and $R_\Omega$ denote,
respectively, the sets of all finitary partial operations and
all finitary relations on the underlying set~$M$ of~$\MB$ that
`commute' with the structure of~$\MB$, and define $\Omega := H_\Omega \cup R_\Omega$.
Each subset $\Delta$ of~$\Omega$ may be interpreted as a signature.
Let $\CZ_\Delta$ denote the category of all Boolean spaces enriched with
continuous/closed structure of signature~$\Delta$, and let $\MT_\Delta$
denote the obvious object in~$\CZ_\Delta$ with underlying set~$M$.
Every alter ego of~$\MB$ is of the form $\MT_\Delta$, for some $\Delta\subseteq\Omega$.
Thus our question becomes:
\begin{itemize}
\item
Given that $\CA$ is fully dualisable, for which signatures $\Delta\subseteq \Omega$ is
$\CA$ fully dualised within $\CZ_\Delta$ via $(\MB,\MT_\Delta)$?
\end{itemize}
Localising to the cogenerators, this question becomes:
\begin{itemize}
\item
Given that some alter ego fully dualises~$\MB$, for which signatures $\Delta\subseteq \Omega$
does $\MT_\Delta$ fully dualise~$\MB$?
\end{itemize}
Here we address the practical problem of recognising those
$\Delta\subseteq \Omega$ for which $\MT_\Delta$ fully dualises~$\MB$.
To explain our approach, we first describe three helpful tools.

\subsubsection*{Tool 1}

The first tool is a syntactic quasi-ordering of signatures
$\Delta\subseteq \Omega$, with $\Omega$ at the top, whose corresponding
equivalence relation~$\equiv$ captures a useful notion of `syntactic equivalence'.
(See Definitions~\ref{def:cloep}--\ref{def:sred} and Lemma~\ref{lem:se-functor}.)
It is known that:
\begin{enumerate}
\item[(a)]
the `fully dualises $\MB$' relation on alter egos is invariant under $\equiv$, and
\item[(b)]
if some $\MT_\Delta$ fully dualises $\MB$, then $\MT_\Omega$ does as well~\cite[5.3]{DPW:galois}.
\end{enumerate}
A~naive expectation is that, more generally, if $\MT_\Delta$ fully dualises~$\MB$
and $\Delta\subseteq \Delta'$, then $\MT_{\Delta'}$ (being at least as rich
as~$\MT_\Delta$) should also fully dualise~$\MB$.

\subsubsection*{Tool 2}

The second tool is `reduction to the finite level'. Once again, fix $\CZ$ and~$\MT$.
Let~$\CA_{\mathrm{fin}}$ and $\CZ_{\mathrm{fin}}$ denote the subcategories of~$\CA$ and~$\CZ$,
respectively, consisting of the objects whose underlying sets are finite.
We say that $\CA$ is fully dualised within $\CZ$ via $(\MB,\MT)$ \emph{at the finite level}
if $\CA_{\mathrm{fin}} \subseteq \Fix\eta$ and
$\IScP\MT\cap \CZ_{\mathrm{fin}} \subseteq \Fix\varepsilon$.
Localising to $\MB$ and~$\MT$, we say that $\MT$ fully dualises $\MB$ at the finite level.
Informally, this means that the dual adjunction satisfies the conditions of
being a full duality at the level of finite objects.

A reasonable intuition is that,
while the full dualisability of a quasivariety~$\CA$ is determined at the infinite level,
if $\CA$ \emph{is} fully dualisable, then the problem of determining which signatures
$\Delta\subseteq \Omega$ give rise to fully dualising alter egos should be determined
at the finite level. In particular, a second naive expectation is that, if some
alter ego fully dualises~$\MB$ and $\MT_\Delta$~fully dualises~$\MB$ at the finite level,
then $\MT_\Delta$ should fully dualise~$\MB$.

\medskip

It turns out that both naive expectations stated above are falsified by examples
(see \cite[5.1]{DPW:galois} and \cite[Thm~1]{DHW:?}),
which explains in part the delicateness of the problem.

\subsubsection*{Tool 3}

The third tool, which will help us to overcome these issues, is the logic of universal
Horn sentences in finitary signatures with partial operations. An easy observation
is that, if two alter egos $\MT_1$ and $\MT_2$ both fully dualise~$\MB$, then the
categories $\IScP{\MT_1}$ and $\IScP{\MT_2}$ are equivalent, as they are both
dually equivalent to~$\CA$.  We will show that, conversely, in certain situations
we can translate a full duality from a known fully dualising alter ego~$\MT_1$
to another alter ego~$\MT_2$ by defining a concrete isomorphism between the
categories $\IScP{\MT_1}$ and $\IScP{\MT_2}$. (In particular, we can define such an
isomorphism whenever both $\MT_1$ and $\MT_2$ fully dualise~$\MB$.)
The logic of universal Horn sentences is used both to articulate the assumptions
that make this work and to define the isomorphism.

We outline the relevant universal Horn logic in Sections~\ref{sec:fulldualities}
and~\ref{sec:logic}. For a more detailed introduction to universal Horn logic
as it applies to the axiomatisation of dual categories,
see Clark, Davey, Haviar, Pitkethly and Talukder~\cite{CDHPT:standard}.

\subsection*{Our results}

We assume that we are given an alter ego~$\MT_1$ that fully dualises a finite
algebra~$\MB$ at the finite level. Our main theorem (Theorem~\ref{thm:trans})
gives necessary and sufficient conditions, organised into three families,
under which an alter ego~$\MT_2$ also fully dualises~$\MB$ at the finite level.
These conditions are a fragment of those in the known characterisation
(Lemma~\ref{lem:fdual}), and depend on both $\MB$ and~$\MT_1$.
In particular, one family of conditions is constructed from a basis for the universal
Horn theory of~$\MT_1$.

\subsubsection*{Self-contained corollaries}

While the statement of the main theorem is rather technical, we use
the theorem to obtain a series of self-contained corollaries.
One corollary gives a new and very natural condition under which
every finite-level full duality lifts to the infinite level.
First, we say that an alter ego $\MT$ is \emph{standard} if the category
$\IScP{\MT}$ consists of all the structures in~$\CZ$ that satisfy the
universal Horn theory of~$\MT$; see Definition~\ref{def:std}.
Here is the corollary (Theorem~\ref{cor:std}):
\begin{itemize}
\item[]
{\itshape Assume that some standard alter ego fully dualises\/~$\MB$.
If an alter ego\/ $\MT$ fully dualises\/~$\MB$ at the finite level\up,
then\/ $\MT$ fully dualises\/~$\MB$ and is also standard.}
\end{itemize}
It follows from this corollary that, for any quasi-primal algebra~$\MB$,
every finite-level full duality based on~$\MB$ lifts to a full duality
(Example~\ref{ex:qp}).

Other corollaries include a new characterisation of the alter egos that
yield a finite-level full duality (Theorem~\ref{thm:fatfl_char})
and a new constructive description of the smallest alter ego that yields a
finite-level full duality (Theorem~\ref{cor:alpha}). We also obtain
the known characterisation (Davey, Pitkethly and Willard~\cite[5.3]{DPW:galois})
of how the structure on an alter ego can be enriched without destroying
a full duality (Theorem~\ref{thm:enrich}).

\subsubsection*{Seminal counterexamples explained}

We use our main theorem to elucidate two important
counter\-examples in the theory of natural dualities:
\begin{itemize}
\item
\emph{The first example of a finite-level full duality that is
not equivalent to~$\MT_\Omega$} (Davey, Haviar and Willard~\cite{DHW:?}).
This example is based on the three-element bounded lattice $\ThB$
and the alter ego $\ThT_h$ defined in Section~\ref{sec:three}.
\item
\emph{The first example of a full duality that is not equivalent
to~$\MT_\Omega$} (Clark, Davey and Willard~\cite{CDW:!}).
This example, which solved a 27-year-old problem from~\cite{DW:old},
is based on a four-element quasi-primal algebra $\QB$
and the alter ego $\QT_0$ defined in Example~\ref{ex:r1}.
\end{itemize}
We give a general algorithm (Algorithm~\ref{cor:alg}) that, given
(i)~an alter ego~$\MT_1$ that fully dualises~$\MB$ at the finite level
(typically, but not necessarily, equivalent to~$\MT_\Omega$), and
(ii)~a finite basis for the universal Horn theory of~$\MT_1$,
produces the smallest alter ego (up to~$\equiv$) that
fully dualises~$\MB$ at the finite level. This algorithm can be applied
to obtain the two examples listed above; see Example~\ref{ex:three}.

\subsubsection*{Different full dualities}

In the final section of the paper, we clarify what it means for two
full dualities based on the same finite algebra~$\MB$ to be `different'.
We show that the concept of `structural embedding' is not categorical
in the concrete dual category $\Fix\varepsilon = \IScP\MT$.
(This is in contrast to the fact that the concept \emph{is} categorical
in the larger category~$\CZ$.) More precisely, if $\MT_1$ and $\MT_2$
are two alter egos, both of which fully dualise~$\MB$, then the two
dual categories $\IScP{\MT_1}$ and $\IScP{\MT_2}$ are necessarily
isomorphic as concrete categories (Lemma~\ref{lem:transfer2}),
but the isomorphism cannot preserve and reflect structural embeddings
unless the signatures of~$\MT_1$ and~$\MT_2$ are equivalent
(Lemma~\ref{lem:se-functor}).

\section{Preliminaries: Full dualities}\label{sec:fulldualities}

In this section, we formalise many of the concepts discussed more informally in the introduction.
For a comprehensive introduction to the theory of natural dualities,
see the Clark--Davey text~\cite{CD:book}.

Fix a finite algebra~$\MB = \langle M; \Sigma\rangle$ and consider the quasivariety
$\CA := \ISP\MB$, that is, the class of all isomorphic copies of subalgebras of arbitrary
powers of~$\MB$. Our conventions are that $\CA$ never contains the empty algebra
and that $\CA$ always contains the one-element algebras (via the zero power);
for other consistent conventions, see~\cite{DPW:galois}.

\begin{itemize}
\item
Let $r$ be an $n$-ary relation on $M$, for some $n\ge 0$.
Then $r$ is said to be \emph{compatible with $\MB$} if it forms
a subalgebra $\rb$ of~$\MB^n$.

\item
Let $h$ be an $n$-ary partial operation on $M$, for some $n\ge 0$.
Then $h$ is said to be \emph{compatible with $\MB$} if
the $(n+1)$-ary relation
\[
\graph h:= \{\, (\vec a, h(\vec a))\mid \vec a \in \dom h\,\}
\]
is compatible with $\MB$, or equivalently, if the $n$-ary relation $r :=\dom h$
is compatible with $\MB$ and $h \colon \rb \to \MB$ is a homomorphism.
\end{itemize}

An \emph{alter ego} of $\MB$ is a topological structure $\MT = \langle M;  H, R, \T\rangle$
with the same underlying set as~$\MB$, where
\begin{itemize}
\item $H$ is a set of partial operations that are compatible with~$\MB$,

\item $R$ is a set of relations that are compatible with~$\MB$, and

\item $\T$ is the discrete topology on $M$.
\end{itemize}
It is common to add a set $G$ of total operations to the signature of $\MT$,
but to simplify the notation, we include total operations in~$H$.

An alter ego $\MT$ is the starting point for creating a potential
dual category~$\CX$ for the quasivariety $\CA = \ISP\MB$.
First, we form the category $\CZ$ whose objects are the
\emph{Boolean structures} of signature~$(H,R)$.
(That is, each member of~$\CZ$ is a topological structure with a Boolean topology and
with continuous partial operations on closed domains and with closed relations.)
The morphisms of~$\CZ$ are continuous structure-preserving maps.
The potential dual category~$\CX$ of~$\CA$ will be a full subcategory of~$\CZ$.

We require the usual concept of induced substructure and the concept of structural embedding:
\begin{itemize}
\item
For $\X,\Y \in \CZ$, we say that $\X$~is an \emph{induced substructure} of~$\Y$
if $X\subseteq Y$, the topology on~$\X$ is the induced subspace topology from~$\Y$,
the relations in~$R^\X$ are the restrictions of those in~$R^\Y$, and the
domains and graphs of the partial operations in~$H^\X$ are the restrictions
of those in~$H^\Y$.
\item
We define a morphism in $\CZ$ to be a \emph{structural embedding} if
\begin{enumerate}
\item[(a)]
it is a homeomorphism from its domain to its range considered
as an induced subspace of its codomain, and
\item[(b)]
it preserves and reflects the relations in the signature and
the domains and graphs of partial operations in the signature.
\end{enumerate}
\end{itemize}
The alter ego~$\MT$ of~$\MB$ induces a pair of contravariant hom-functors
$D \colon \CA \to \CZ$ and $E \colon \CZ \to \CA$, and a pair of natural
transformations $\eta \colon \id_\CA \to ED$ and
$\varepsilon \colon \id_\CZ \to DE$.
The hom-functors $D$ and $E$ are given on objects by
\begin{align*}
D(\A) &:= \text{the induced substructure of $\MT^A$ with underlying set $\rhom\CA(\A,\MB)$}\\
E(\Z) &:= \text{the subalgebra of $\MB^Z$ with underlying set $\rhom\CZ(\Z,\MT)$}
\end{align*}
for all $\A\in \CA$ and $\Z\in \CZ$. The compatibility between $\MB$ and $\MT$
guarantees that these hom-functors are well defined. The natural transformations $\eta$
and $\varepsilon$ are given by evaluation: for all $\A\in \CA$,
the homomorphism $\eta_\A \colon \A \to ED(\A)$ is defined by
\[
\eta_\A(a)(x) := x(a), \quad\text{for all } a\in A \text{ and }  x\in \rhom\CA(\A,\MB),
\]
and, for all $\Z\in \CZ$, the morphism $\varepsilon_\Z \colon \Z \to DE(\Z)$ is defined by
\[
\varepsilon_\Z(z)(u) := u(z), \quad\text{for all } z\in Z \text{ and } u\in \rhom\CZ(\Z,\MT).
\]
It is easily seen that $\eta_\A \colon \A \to ED(\A)$ is an embedding, for all $\A\in \CA$,
since $\CA = \ISP\MB$.

We form the \emph{topological quasivariety} $\CX := \IScP\MT$ consisting
of all isomorphic copies of topologically closed induced substructures
of non-zero powers of~$\MT$. Our conventions are that $\CX$~contains the
empty structure if and only if $H$~contains no nullary operations, and that
$\CX$~contains a one-element structure if and only if it appears as an
induced substructure of~$\MT$. It is easily seen that
$\varepsilon_\X \colon \X \to DE(\X)$ is a structural embedding, for all $\X\in \CX$.

The basic concepts, localised to the pair $(\MB,\MT)$, are defined as follows:
\begin{enumerate}
\item
\emph{$\MT$ dualises $\MB$} [\emph{at the finite level}]
if the embedding $\eta_\A \colon \A \to ED(\A)$ is an isomorphism,
for each [finite] algebra $\A \in \CA$.
\item
\emph{$\MT$ fully dualises $\MB$} [\emph{at the finite level}]
if, in addition to~(1), the structural embedding
$\varepsilon_\X \colon \X \to DE(\X)$ is an isomorphism,
for each [finite] structure $\X \in \CX$.
\item
\emph{$\MT$ strongly dualises $\MB$} [\emph{at the finite level}]
if, in addition to~(1) and~(2), the alter ego $\MT$ is injective
with respect to structural embeddings among the [finite] structures in~$\CX$.
\end{enumerate}
Using the $\Fix$ notation of the introduction:
\begin{itemize}
\item
If $\MT$ dualises $\MB$, then $\Fix\eta = \CA$, and hence
$\CA$~is dually equivalent to a full subcategory of~$\CX$.
\item
If $\MT$ fully dualises $\MB$, then $\Fix\eta = \CA$ and $\Fix\varepsilon = \CX$,
and hence $\CA$~is dually equivalent to~$\CX$.
\end{itemize}

The following basic lemma will allow us to create new full
dualities from old ones. There are two
versions of this lemma: the phrases in square brackets can
be either included or deleted.

\begin{lemma:transfer}\label{lem:transfer}
Let\/ $\MT_1$ and $\MT_2$ be alter egos of a finite
algebra\/~$\MB$. For $i\in \{1,2\}$\up, define $\CX_i := \IScP{\MT_i}$.
Assume that\/ $\MT_1$ fully dualises $\MB$
\up[at the finite level\up]. Then $\MT_2$ also fully dualises
$\MB$ \up[at the finite level\up] provided the following two
conditions hold\up:
\begin{enumerate}
 \item $\MT_2$ dualises $\MB$ \up[at the finite level\up]\up;

 \item for each \up[finite\up] structure $\X$ in $\CX_2$\up, there
     is a structure $\X'$ in $\CX_1$ on the same underlying set as
     $\X$ such that\/ $\rhom{\CX_2}(\X,\MT_2) = \rhom{\CX_1}(\X',\MT_1)$.
\end{enumerate}
\end{lemma:transfer}
\begin{proof}
Define $\CA := \ISP \MB$ and, for $i\in \{1, 2\}$, let
$D_i\colon \CA \to \CX_i$ and $E_i\colon \CX_i \to \CA$ be the
hom-functors induced by $\MB$ and $\MT_i$.
Assume that (1) and (2) hold. Let $\X$ be a [finite] structure
in~$\CX_2$. We just need to show that $\varepsilon_\X \colon \X
\to D_2E_2(\X)$ is surjective, that is, we need to show that
every homomorphism $u \colon E_2(\X) \to \MB$ is given by
evaluation.

By~(2), we have $\rhom{\CX_2}(\X,\MT_2) = \rhom{\CX_1}(\X',\MT_1)$.
Thus $E_2(\X) = E_1(\X')$ in~$\CA$. As $\MT_1$ fully
dualises $\MB$ [at the finite level], each homomorphism $u
\colon E_1(\X') \to \MB$ is given by evaluation. It follows
at once that each homomorphism $u \colon E_2(\X) \to
\MB$ is given by evaluation.
\end{proof}

We close this section with a brief discussion of universal Horn sentences
and their role in attempts to axiomatise dual categories.

Fix a signature $(H,R)$ of finitary partial-operation and relation
symbols. We define a \emph{universal Horn sentence} (\emph{uH-sentence},
for short) in the language of $(H,R)$ to be a first-order sentence of
the form
\[
\forall \vec v\, \Bigl[ \Bigl(\bigand_{i = 1}^\nu \alpha_i(\vec v)\Bigr) \to
\gamma(\vec v) \Bigr],
\]
for some $\nu \ge 0$, where each $\alpha_i(\vec v)$ is an
atomic formula and $\gamma(\vec v)$ is either an atomic formula
or~$\bot$. (Note that, if $\nu = 0$, then we have a sentence of the form
$\forall \vec v\, \gamma(\vec v)$.)

\begin{definition}[{\cite{CDHPT:standard}}]\label{def:std}
Let $\MT = \langle M; H, R, \T\rangle$ be an alter ego of a finite algebra~$\MB$,
and let $\CZ$ be the associated category of Boolean structures of signature $(H,R)$.
The potential dual category $\CX = \IScP\MT$ is always contained in the
category~$\CY$ of all Boolean models of the uH-theory of~$\MT$. That is,
\[
 \CX \subseteq \CY := \bigl\{\, \Y \in \CZ \bigm| \Y \models \ThuH {\MT} \,\bigr\},
\]
where $\ThuH {\MT}$ denotes the set of all uH-sentences true in~$\MT$.

If the two categories $\CX$ and $\CY$ are equal, then we say that the
alter ego~$\MT$ is \emph{standard}. For example, the discrete semilattice
$\ST = \langle \{0,1\}; \vee, \T \rangle$ (from Hofmann--Mislove--Stralka duality)
is standard~\cite{HMS:semi}, but the discrete chain $\TwT = \langle \{0,1\}; \le, \T \rangle$
(from Priestley duality) is not standard~\cite{Stralka}.

Note that we always have $\CX_{\mathrm{fin}} = \CY_{\mathrm{fin}}$.
That is, the finite structures in~$\CX$ are precisely the finite Boolean
models of the uH-theory of~$\MT$; see \cite[pp.\ 861--862]{CDHPT:standard}.
\end{definition}

\section{Preliminaries: Comparing alter egos}\label{sec:alteregos}

This section gives the more specific background theory that we require.
We start by defining the `structural reduct' quasi-order on the alter
egos of a finite algebra~$\MB$; see~\cite{DHW:struct,DPW:galois}.
This is the natural generalisation from algebras to structures of the
`term reduct' quasi-order.

\begin{definition}[{\cite{CD:book}}]\label{def:cloep}
Given any alter ego~$\MT = \langle M; H, R, \T \rangle$ of~$\MB$,
we define $\clo \MT$ to be the \emph{enriched partial clone} on~$M$ generated
by~$H$, that is, the smallest set of non-empty partial operations on $M$ that contains~$H$
and the projections, $\pi_i \colon M^n \to M$ for all $n\ge 1$ and~$i\le n$,
and is closed under composition (when the composite has non-empty domain). This corresponds
to the usual definition of partial clone, except that we exclude empty domains and
we enrich the partial clone by allowing nullary operations.
\end{definition}

\begin{definition}[\cite{DPW:galois}]\label{def:relca}
Let~$\MT = \langle M; H, R, \T \rangle$ be an alter ego of~$\MB$ and let $k,n\ge 0$.
\begin{itemize}
\item
We shall call a conjunction of atomic formul\ae\
$\Psi(\vec v) = [\psi_1(\vec v) \And \dotsb \And \psi_k(\vec v)]$
a \emph{conjunct-atomic formula}.

\item
We say that a non-empty $n$-ary relation $r$
on $M$ is \emph{conjunct-atomic definable} from $\MT$ if it is
described in $\MT$ by an $n$-variable conjunct-atomic formula
$\Psi(\vec v)$ in the language of~$\MT$, that is, if
\[
r = \bigl\{\, (a_1,\dotsc,a_n) \in M^n \bigm| \Psi(a_1,\dotsc,a_n)
\text{ is true in $\MT$} \,\bigr\}.
\]

\item
We define $\cadef \MT$ to be the set of all relations on $M$ that are
conjunct-atomic definable from~$\MT$.
\end{itemize}
\end{definition}

\begin{definition}[{\cite{DHW:struct}}]\label{def:sred}
Let $\MT_1 = \langle M; H_1, R_1, \T \rangle$ and $\MT_2 = \langle M;
H_2, R_2, \T \rangle$ be alter egos of~$\MB$. Then we say that
$\MT_1$ is a \emph{structural reduct} of~$\MT_2$ if
 \begin{enumerate}
 \item[(a)] each partial operation in $H_1$ has an extension in $\clo {\MT_2}$, and
 \item[(b)] each relation in $R_1 \cup \dom{H_1}$ belongs to $\cadef {\MT_2}$.
\end{enumerate}
We say that $\MT_1$ and $\MT_2$ are \emph{structurally
equivalent} if each is a structural reduct of the other.
\end{definition}

Under the `structural reduct' quasi-order, the alter egos of~$\MB$ form a
doubly algebraic lattice~$\mathcal A_\MB$; see~\cite[2.6]{DPW:galois}.
The top element of this lattice is represented by the \emph{top alter ego}
of~$\MB$, which we denote by $\MT_\Omega= \langle M; H_\Omega, R_\Omega, \T \rangle$,
where
\begin{itemize}
\item $H_\Omega$ is the set of all partial operations that are compatible with~$\MB$, and
\item $R_\Omega$ is the set of all relations that are compatible with~$\MB$.
\end{itemize}
With the help of the following definitions and lemmas, we will be able
to describe how the various flavours of duality occur within the lattice~$\mathcal A_\MB$;
see Facts~\ref{fac:galois}.

\begin{definition}
Let $r$ be an $n$-ary relation compatible with~$\MB$, for some $n \ge 0$,
and let $\rb$ be the subalgebra of~$\MB^n$ with $r$ as underlying set.
\begin{itemize}
 \item We say that the relation $r$ is \emph{hom-minimal}
     if every homomorphism from $\rb$ to $\MB$ is a
     projection (see~\cite{DHW:struct}).
 \item We say that $\MT$ is \emph{operationally rich} at
     $r$ if every compatible partial operation on $\MB$ with
     domain $r$ has an extension in $\clo \MT$.
\end{itemize}
\end{definition}

\begin{lemma:dual}[{\cite[4.1]{DPW:galois}}]\label{lem:dual}
Let\/ $\MT$ be an alter ego of a finite algebra~$\MB$. Then $\MT$
dualises $\MB$ at the finite level if and only if
every hom-minimal relation on $\MB$ belongs to $\cadef\MT$.
\end{lemma:dual}

\begin{remark}\label{rem:lift}
If a finite algebra $\MB$ has an alter ego that yields a duality,
then every finite-level duality based on~$\MB$ lifts to the infinite level;
see~\cite[p.~19]{DPW:galois}. Note that the same is not true in general
for full duality~\cite{DHW:?}; see Lemma~\ref{lem:thth} and Remark~\ref{rem:thth}.
\end{remark}

We shall use the description of finite-level full duality
provided by the following lemma. In fact, our main theorem will
allow us to give a more refined version of this lemma
(see Theorem~\ref{thm:fatfl_char}).

\begin{lemma:fdual}[{\cite[4.3]{DPW:galois}}]\label{lem:fdual}
Let\/ $\MT$ be an alter ego of a finite algebra~$\MB$. Then
$\MT$ fully dualises $\MB$ at the finite level if and only if
\begin{enumerate}
 \item[(a)] every hom-minimal relation on $\MB$ belongs to $\cadef\MT$\up, and
 \item[(b)] $\MT$ is operationally rich at each relation in $\cadef \MT$.
\end{enumerate}
\end{lemma:fdual}

\begin{remark}
Since $M^n\in \cadef\MT$, for all $n\ge 0$, it follows from~(b) above that every
compatible total operation (that is, every homomorphism $g\colon \MB^n \to \MB$)
belongs to $\clo \MT$. In particular, every element of $M$ that forms a one-element
subalgebra of $\MB$ must be the value of a nullary operation in $\clo \MT$.
\end{remark}

\begin{facts}\label{fac:galois}
The following facts about the lattice of alter egos $\mathcal A_\MB$ are proved
in~\cite{DPW:galois}; see Figure~\ref{fig:AM}.
\begin{enumerate}
\item By the Duality Lemma~\ref{lem:dual},
the alter egos that dualise $\MB$ at the finite level
form a principal filter of~$\mathcal A_\MB$.

\item It follows from the Full Duality Lemma~\ref{lem:fdual} that,
under the `structural reduct' quasi-order, there is a smallest
alter ego~$\MT_\alpha$ that fully dualises $\MB$ at the finite level;
see~\cite[4.4]{DPW:galois}.

\item The alter egos that fully dualise $\MB$ at the finite level form a complete
sublattice $\mathcal F_\MB$ of $\mathcal A_\MB$ and those that fully dualise $\MB$
form an up-set of~$\mathcal F_\MB$; see~\cite[5.5]{DPW:galois}.

\item An alter ego strongly dualises $\MB$ at the finite level if and only if it is
structurally equivalent to the top alter ego $\MT_\Omega$, and so
there is essentially only one candidate for a strong duality; see~\cite[4.6]{DPW:galois}.
\end{enumerate}
\end{facts}

\begin{figure}[t]
\begin{tikzpicture}
  \begin{scope}
    \clip (0,1.5) circle [x radius=1.25cm, y radius=1.5cm];
    \clip (-4,4) -- (0,2) -- (4,4) -- (-4,4);
    \fill[black!20] (-4,0) rectangle (4,4);
  \end{scope}
  \node[label,anchor=base] at (-1.5,0) {$\mathcal A_\MB$};
  \draw[name path=ellipse] (0,1.5) circle [x radius=1.25cm, y radius=1.5cm];
  \node[unshaded] (b) at (0,0) {}; 
  \node[unshaded] (d) at (0,1) {}; 
  \node[shaded] (f) at (0,2) {}; 
  \node[label] at (f) [below left=0pt and 1pt] {$\MT_\alpha$};
  \node[shaded] (s) at (0,3) {}; 
  \node[label] at (s) [above left=5pt and 1pt] {$\MT_\Omega$};
  \node[label,anchor=north west] (dl) at (2,1.5) {\parbox{4cm}{Filter of all\\ finite-level dualities}};
  \node[label,anchor=north west] (fl) at (2,2.75) {\parbox{4cm}{Filter includes all\\ finite-level full dualities}};
  \node[label,anchor=north west] (sl) at (2,3.5) {\parbox{4cm}{Finite-level strong duality}};
  \draw[straight,thin] (dl.west) to ($(d)+(0,0.5)$);
  \draw[straight,thin] (fl.west) to ($(f)+(0,0.5)$);
  \draw[straight,thin] (sl.west) to (s);
  \path [name path=line1] (d) -- ++(-2,1);
  \path [name path=line2] (d) -- ++(2,1);
  \draw [densely dashed,name intersections={of=ellipse and line1}] (d) -- (intersection-1);
  \draw [densely dashed,name intersections={of=ellipse and line2}] (d) -- (intersection-1);
  \path [name path=line3] (f) -- ++(-2,1);
  \path [name path=line4] (f) -- ++(2,1);
  \draw [densely dashed,name intersections={of=ellipse and line3}] (f) -- (intersection-1);
  \draw [densely dashed,name intersections={of=ellipse and line4}] (f) -- (intersection-1);
\end{tikzpicture}
\caption{The lattice of alter egos of a finite algebra $\MB$}\label{fig:AM}
\end{figure}
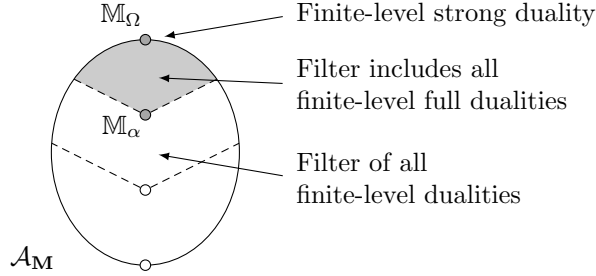

\section{Motivating example}\label{sec:three}

In this section, we illustrate the general idea behind the
proof of our New-from-old Theorem~\ref{thm:trans} using the
three-element bounded lattice
\[
\ThB = \langle \{0,a,1\}; \vee, \wedge, 0, 1\rangle,
\]
which has played a seminal role as an example in the theory of
natural dualities.

We use the four compatible partial operations on $\ThB$ shown in
Figure~\ref{fig:fgh}: the two unary operations $f$ and $g$, and the two binary
partial operations $\sigma$ and $h$. The alter ego $\ThT := \langle \{0,a,1\};
f,g,\T\rangle$ dualises~$\ThB$, and can be obtained from
Priestley duality using general `duality transfer'
techniques (Davey~\cite{D:siena}). The alter ego
\[
\ThT_\sigma := \langle \{0,a,1\}; f,g,\sigma,\T\rangle
\]
strongly dualises~$\ThB$, and can be obtained from Priestley
duality using general `strong duality transfer'
techniques (Davey and Haviar~\cite{DH:aids}). Note that, since $\ThT_\sigma$
strongly dualises $\ThB$ at the finite level, it must be
equivalent to the top alter ego of~$\ThB$;
see Facts~\ref{fac:galois}(4).

\begin{figure}[t]
\begin{tikzpicture}
   \begin{scope}
      \node[anchor=base] at (-0.875,0.5) {$f$};
      \node[unshaded] (0) at (0,0) {}; \node[label] at (0) [right=5pt] {$0$};
      \node[unshaded] (a) at (0,1) {}; \node[label] at (a) [right=5pt] {$a$};
      \node[unshaded] (1) at (0,2) {}; \node[label] at (1) [right=5pt] {$1$};
      \draw[order] (0) to (a); \draw[order] (a) to (1);
      \draw[loopy] (0) to [out=-155, in=-195] (0);
      \draw[curvy] (a) to [bend right] (0);
      \draw[loopy] (1) to [out=-155, in=-195] (1);
   \end{scope}
   \begin{scope}[xshift=3cm]
      \node[anchor=base] at (-0.875,0.5) {$g$};
      \node[unshaded] (0) at (0,0) {}; \node[label] at (0) [right=5pt] {$0$};
      \node[unshaded] (a) at (0,1) {}; \node[label] at (a) [right=5pt] {$a$};
      \node[unshaded] (1) at (0,2) {}; \node[label] at (1) [right=5pt] {$1$};
      \draw[order] (0) to (a); \draw[order] (a) to (1);
      \draw[loopy] (0) to [out=-155, in=-195] (0);
      \draw[curvy] (a) to [bend left] (1);
      \draw[loopy] (1) to [out=-155, in=-195] (1);
   \end{scope}
   \begin{scope}[xshift=6cm]
      \node[anchor=base] at (-1.5,0.5) {$\sigma$};
      \node[unshaded] (00) at (0,0) {}; \node[label] at (00) [left=5pt] {$(0,0)$};
      \node[unshaded] (01) at (0,1) {}; \node[label] at (01) [left=5pt] {$(0,1)$};
      \node[unshaded] (11) at (0,2) {}; \node[label] at (11) [left=5pt] {$(1,1)$};
      \node[unshaded] (0) at (1.125,0) {}; \node[label] at (0) [right=5pt] {$0$};
      \node[unshaded] (a) at (1.125,1) {}; \node[label] at (a) [right=5pt] {$a$};
      \node[unshaded] (1) at (1.125,2) {}; \node[label] at (1) [right=5pt] {$1$};
      \draw[order] (00) to (01); \draw[order] (01) to (11);
      \draw[order] (0) to (a); \draw[order] (a) to (1);
      \draw[straight] (00) to (0);
      \draw[straight] (01) to (a);
      \draw[straight] (11) to (1);
   \end{scope}
   \begin{scope}[xshift=10cm]
      \node[anchor=base] at (-1.5,0.5) {$h$};
      \node[unshaded] (00) at (0,0) {}; \node[label] at (00) [left=5pt] {$(0,0)$};
      \node[unshaded] (0a) at (0,0.667) {}; \node[label] at (0a) [left=5pt] {$(0,a)$};
      \node[unshaded] (a1) at (0,1.333) {}; \node[label] at (a1) [left=5pt] {$(a,1)$};
      \node[unshaded] (11) at (0,2) {}; \node[label] at (11) [left=5pt] {$(1,1)$};
      \node[unshaded] (0) at (1.125,0) {}; \node[label] at (0) [right=5pt] {$0$};
      \node[unshaded] (a) at (1.125,1) {}; \node[label] at (a) [right=5pt] {$a$};
      \node[unshaded] (1) at (1.125,2) {}; \node[label] at (1) [right=5pt] {$1$};
      \draw[order] (00) to (0a); \draw[order] (0a) to (a1); \draw[order] (a1) to (11);
      \draw[order] (0) to (a); \draw[order] (a) to (1);
      \draw[straight] (00) to (0);
      \draw[straight] (0a) to (a);
      \draw[straight] (a1) to (a);
      \draw[straight] (11) to (1);
   \end{scope}
\end{tikzpicture}
\caption{The compatible partial operations $f$, $g$, $\sigma$ and $h$ on $\ThB$}\label{fig:fgh}
\end{figure}
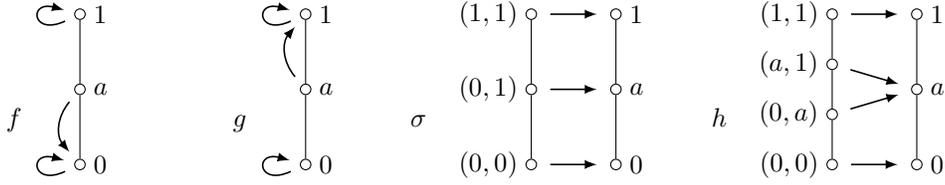

The first example of a finite-level full but not strong duality
(given by Davey, Haviar and Willard~\cite{DHW:?}) was based on
the alter ego
\[
\ThT_h := \langle \{0,a,1\}; f,g,h, \T\rangle.
\]
This alter ego was not found using general techniques.
Later in this paper, we shall give a general `full duality
transfer' technique that will allow us to obtain this alter ego
in a natural way from~$\ThT_\sigma$; see Example~\ref{ex:three}.

In this section, we give a new proof that $\ThT_h$ fully
dualises $\ThB$ at the finite level. We will show how to
transfer the finite-level full duality down from $\ThT_\sigma$
to $\ThT_h$ by using a basis for the universal Horn theory
of~$\ThT_\sigma$.

We want to apply the New-from-old Lemma~\ref{lem:transfer}. So
we need a way to enrich each finite structure $\X$ in $\CX_h :=
\IScP{\ThT_h}$ into a structure $\X^\sharp$ in $\CX_\sigma :=
\IScP{\ThT_\sigma}$. We will check membership of~$\CX_\sigma$
syntactically: we know that at the finite level $\CX_\sigma$ is
axiomatised by the universal Horn theory of~$\ThT_\sigma$.

\begin{definition}\label{def:sharp}
Let $\X = \langle X; f^\X,g^\X,h^\X,\T^\X\rangle$ be a finite
structure in~$\CX_h$. We want to define a structure $\X^\sharp$
of the same signature as~$\ThT_\sigma$. The binary partial operation
$\sigma$ is described in $\ThT_\sigma$ by the sentence
\[
\forall uvw \,\bigl[{\sigma(u,v) = w}\, \leftrightarrow \bigl(f(w) =
u \And g(w) = v\bigr)\bigr],
\]
which is logically equivalent to a conjunction of uH-sentences.
So we would like to define the partial operation $\psigma \X$
on $X$ by
\[
\graph {\psigma \X} := \bigl\{\, (x,y,z) \in X^3 \bigm| f^\X(z) =
x \And g^\X(z) = y \,\bigr\}.
\]
As the endomorphisms $f$ and $g$ separate the elements
of~$\ThB$, the uH-sentence
\[
\forall uv \,\bigl[\bigl( f(u) = f(v) \And g(u) = g(v)\bigr) \to
\,{u = v} \bigr]
\]
holds in~$\ThT_h$ and therefore in~$\X$. This tells us that
$\graph {\psigma \X}$ really is the graph of a binary partial operation
on $X$ (possibly an empty operation). So we can define
\[
\X^\sharp := \langle X; f^\X,g^\X,\psigma \X,\T^\X\rangle,
\]
and $\X^\sharp$ is a (discrete) Boolean structure of the same
signature as~$\ThT_\sigma$.
\end{definition}

\begin{remark}
The operation $\psigma \X$ defined above has a natural
interpretation in the case that $\X$ is a concrete structure
in~$\CX_h$. Assume that $\X \le (\ThT_h)^k$, for some $k
> 0$. Then we can impose the ternary relation
$\graph \sigma$ coordinate-wise on the set~$X$. The operation
$\psigma \X$ is defined so that $\graph {\psigma \X} = \graph
\sigma^X$. Thus $\psigma \X$ is the maximum coordinate-wise
extension of $\sigma$ to~$X$.
\end{remark}

\begin{lemma}\label{lem:sharp1}
Let\/ $\X$ be a finite structure in\/~$\CX_h := \IScP {\ThT_h}$.
Then the structure $\X^\sharp$ defined in~\ref{def:sharp}
belongs to $\CX_\sigma := \IScP {\ThT_\sigma}$.
\end{lemma}
\begin{proof}
As the structure $\X^\sharp$ is finite, we just have to check
it is a model of the universal Horn theory of~$\ThT_\sigma$.
The basis for $\ThuH {\ThT_\sigma}$ given by Clark, Davey,
Haviar, Pitkethly and Talukder~\cite[3.6]{CDHPT:standard} can be
reduced to the following set of sentences:
\begin{enumerate}
 \item $\forall v\,\bigl[f(v) = f(f(v)) = g(f(v)) \And g(v) =
     f(g(v)) = g(g(v))\bigr]$;
 \item $\forall uvw\,\bigl[\bigl(f(w) = u \And  g(w) =
     v\bigr) \leftrightarrow \sigma(u,v) = w\bigr]$;
 \item $\forall uv\,\bigl[\bigl(\sigma(u,v) = \sigma(u,v) \And
     \sigma(v,u) = \sigma(v,u) \bigr) \to u = v\bigr]$;
 \item $\forall uvw\,\bigl[\bigl(\sigma(u,v) = \sigma(u,v) \And
     \sigma(v,w) = \sigma(v,w) \bigr) \to \sigma(u,w) = \sigma(u,w)\bigr]$.
\end{enumerate}
Since sentence~(1) is in the language of $f$ and~$g$, it is
also part of the uH-theory of~$\ThT_h$. So $\X^\sharp$
satisfies~(1), as $\X \in \CX_h$. Sentence (2) holds in
$\X^\sharp$ by construction. Sentence (3) can be transformed
into a uH-sentence in the language of $f$ and~$g$, using
sentence~(2):
\[
\forall uvxy\,\bigl[\bigl(f(x) = u \And g(x) = v \And f(y) = v \And g(y) = u\bigr)
\to u = v\bigr].
\]
So $\X^\sharp$ satisfies~(3), again as $\X \in \CX_h$.

When sentence~(4) is translated into the language of $f$
and~$g$, it becomes
\begin{multline*}
 \phi := \forall uvwxy\,\bigl[
 \bigl(f(x) = u \And g(x) = v \And f(y) = v \And g(y) = w \bigr)
 \to \\
 \exists z \bigl( f(z) = u \And g(z) = w \bigr)
\bigr],
\end{multline*}
which is not a uH-sentence. But we can overcome this problem
using the partial operation~$h$. It is easy to check that $\ThT_h$
satisfies the sentences
\begin{enumerate}
 \item[(5)] $\forall xy\,\bigl[g(x) = f(y) \to h(x,y)=h(x,y)\bigr]$,
 \item[(6)] $\forall xy\,\bigl[h(x,y)=h(x,y) \to f(h(x,y)) = f(x)\bigr]$,
 \item[(7)] $\forall xy\,\bigl[h(x,y)=h(x,y) \to g(h(x,y)) = g(y)\bigr]$.
\end{enumerate}
It follows that $\ThT_h$ satisfies
\begin{multline*}
 \psi := \forall uvwxy\,\bigl[
 \bigl(f(x) = u \And g(x) = v \And f(y) = v \And g(y) = w \bigr) \to \\
 \bigl( f(h(x,y)) = u \And g(h(x,y)) = w \bigr) \bigr],
\end{multline*}
which is logically equivalent to a conjunction of uH-sentences.
Since $\psi \vdash \phi$, it follows that $\X \models \phi$ and
therefore that $\X^\sharp$ satisfies~(4).
\end{proof}

The original proof that $\ThT_h$ fully dualises $\ThB$ at the
finite level piggybacked on Priestley duality. We obtain a more
`generalisable' proof by piggybacking on the strong duality
given by~$\ThT_\sigma$.

\begin{lemma}[\cite{DHW:?}]\label{lem:thth}
The alter ego $\ThT_h := \langle \{0,a,1\}; f,g,h, \T \rangle$
fully dualises the bounded lattice $\ThB$ at the finite level.
\end{lemma}
\begin{proof}
We use the fact that $\ThB$ is dualised by $\ThT := \langle
\{0,a,1\}; f,g, \T \rangle$ and strongly dualised
by~$\ThT_\sigma := \langle \{0,a,1\}; f,g, \sigma, \T \rangle$.
We shall establish conditions~(1) and~(2) of the New-from-old
Lemma~\ref{lem:transfer}, with $\MT_1 = \ThT_\sigma$ and $\MT_2
= \ThT_h$.

Since $\ThT$ dualises~$\ThB$, so does~$\ThT_h$. Now let $\X$ be
a finite structure in~$\CX_h$, and construct the structure
$\X^\sharp$ as in Definition~\ref{def:sharp}. We know that
$\X^\sharp \in \CX_\sigma$, by the previous lemma. It remains
to check that $\rhom{\CX_h}(\X,\ThT_h) =
\rhom{\CX_\sigma}(\X^\sharp,\ThT_\sigma)$.

Define $\CX := \IScP {\ThT}$. Let $\X_\flat$ denote the common reduct of $\X$ and $\X^\sharp$
to the language of~$\ThT$; thus $\X_\flat \in \CX$. Consider a morphism $\mu \colon
\X_\flat \to \ThT$. The construction of~$\X^\sharp$ ensures
that $\mu \colon \X^\sharp \to \ThT_\sigma$ is a morphism.
Since $\graph h \in \cadef {\ThT}$, via the sentence
\[
\forall uvw\, \bigl[{h(u,v) = w} \leftrightarrow  \bigl( f(u) = f(w)
     \And g(v) = g(w) \And g(u) = f(v) \bigr)\bigr],
\]
and since $\X \models \ThuH{\ThT_h}$, we also know that $\mu
\colon \X \to \ThT_h$ is a morphism. Thus
$\rhom{\CX_h}(\X,\ThT_h) = \rhom\CX(\X_\flat, \ThT) =
\rhom{\CX_\sigma}(\X^\sharp,\ThT_\sigma)$, as required.
\end{proof}

\begin{remark}\label{rem:thth}
We know that $\ThT_h$ does not fully dualise~$\ThB$~\cite{DHW:?}.
So this proof must break down somewhere for infinite structures in~$\CX_h$.
For \emph{any} structure $\X \in \CX_h$, we can construct $\X^\sharp$ as in
Definition~\ref{def:sharp}, and we can show that $\X^\sharp$ is
a Boolean model of the uH-theory of~$\ThT_\sigma$. However,
this does not imply that the structure $\X^\sharp$ belongs
to~$\CX_\sigma$, because the alter ego $\ThT_\sigma$ is not
standard~\cite[3.5]{CDHPT:standard}. The connection between
full dualities and standardness is explored in Section~\ref{sec:thm}.
\end{remark}

\section{Background uH-logic for the general case}\label{sec:logic}

The proof of our New-from-old Theorem~\ref{thm:trans}
generalises the proof in the previous section: we transfer a
[finite-level] full duality from one alter ego $\MT_1$ to
another alter ego~$\MT_2$ by using a basis for $\ThuH {\MT_1}$.
In this section, we present the required background uH-logic,
all of which is well known and elementary, except perhaps for
the `partial operations' twist.

Consider a uH-sentence $\forall \vec v \, \bigl[ \bigl(\bigand_{i=1}^\nu
\alpha_i(\vec v) \bigr) \to \gamma(\vec v)\bigr]$ in the language of $(H,R)$.
We call $\bigand_{i=1}^\nu \alpha_i(\vec v)$ the \emph{premise} or
\emph{hypothesis} of the sentence and $\gamma(\vec v)$ the \emph{conclusion}.
We identify a particularly simple form of uH-sentence.

\begin{definition}\label{def:pure}\
\begin{enumerate}
 \item
A formula $\alpha$ is \emph{hypothetically pure} if it has one of the following forms:
\begin{enumerate}
 \item $r(v_{i_1},\dotsc,v_{i_n})$, for some $r \in R$, or
 \item $h(v_{i_1},\dotsc,v_{i_n})=v_{i_0}$, for some $h \in H$,
\end{enumerate}
where $v_{i_0}, v_{i_1}, \dotsc, v_{i_n}$ are variables (not
necessarily distinct).

 \item
A formula $\gamma$ is \emph{conclusively pure} if it has one of the following forms:
\begin{enumerate}
 \item $r(v_{i_1},\dotsc,v_{i_n})$, for some $r \in R$,
 \item $h(v_{i_1},\dotsc,v_{i_n}) =
     h(v_{i_1},\dotsc,v_{i_n})$, for some $h \in H$,
 \item $u=v$, or
 \item $\bot$,
\end{enumerate}
where $u,v,v_{i_1}, \dotsc, v_{i_n}$ are variables (not necessarily distinct).

 \item
A uH-sentence $\forall \vec v \, \bigl[ \bigl(\bigand_{i=1}^\nu
\alpha_i(\vec v) \bigr) \to \gamma(\vec v)\bigr]$ is \emph{pure}
if
\begin{enumerate}
 \item each $\alpha_i$ in the premise is
     hypothetically pure, and
 \item the conclusion $\gamma$ is conclusively pure.
\end{enumerate}
\end{enumerate}
\end{definition}

\begin{lemma}\label{lem:pure}
Every uH-sentence in a language allowing partial-operation
symbols is logically equivalent to a conjunction of pure
uH-sentences.
\end{lemma}
\begin{proof}
We show how to transform an impure uH-sentence $\sigma$ into a
finite number of new uH-sentences, each of which is nearer to
being pure than $\sigma$ (according to some appropriate
well-founded measure).  By recursively applying this process to
each new sentence obtained, we ultimately obtain a finite set
of pure uH-sentences whose conjunction is logically equivalent
to~$\sigma$.

For each transformation in the following list (other than the first transformation),
we (a)~state the relevant logical equivalence, and (b)~display
the new uH-sentence or sentences obtained from $\sigma$ by
using this equivalence and then applying standard prenex
operations. Variables $w$, $w'$ and $\vec w = (w_1,\dotsc,w_n)$
appearing in the statements of the transformations are assumed
not to occur in~$\sigma$.

Given an impure uH-sentence $\sigma = \forall \vec v \,\bigl[
\bigl( \bigand_i \alpha_i(\vec v)\bigr) \to \gamma(\vec v)\bigr]$,
do the following.

\begin{newlist}
 \item[(0)]
If some $\alpha_k$ is of the form $u=v$, for variables $u$ and~$v$,
then remove $\alpha_k$ from~$\sigma$. If the variables $u$ and $v$ are distinct,
then also remove $\forall v$ and replace~$v$ by~$u$ throughout the resulting formula.
 \item[(1)]
Else if some $\alpha_k$ is of the form $r(t_1,\dotsc,t_n)$, with
some $t_\ell$ not a variable:
\begin{enumerate}
\item[(a)] use $\alpha_k \equiv \exists \vec w\, \bigl[
    r(\vec w) \And \bigl(\bigand_{j=1}^n t_j = w_j\bigr) \bigr]$;
\item[(b)] replace $\sigma$ by $\forall \vec v \vec w\, \bigl[
    \bigl(\bigl( \bigand_{i\ne k} \alpha_i(\vec v) \bigr)\And
    r(\vec w)\And \bigl(\bigand_{j=1}^n t_j=w_j\bigr)\bigr)
    \to \gamma(\vec v)\bigr]$.
\end{enumerate}

 \item[(2)]
Else if some $\alpha_k$ is of the form $s = t$, with $t$ not a
variable:
\begin{enumerate}
\item[(a)] use $\alpha_k \equiv\exists w \bigl[ s = w \And
    t = w\bigr]$;
\item[(b)] replace $\sigma$ by $\forall \vec v w \bigl[
    \bigl(\bigl( \bigand_{i\ne  k} \alpha_i(\vec v) \bigr)\And
    s=w\And t=w \bigr) \to \gamma(\vec v)\bigr]$.
\end{enumerate}

 \item[(3)]
Else if some $\alpha_k$ is of the form $h(t_1,\dotsc,t_n) =
v_m$, with some $t_\ell$ not a variable:
\begin{enumerate}
\item[(a)] use $\alpha_k \equiv \exists \vec w\, \bigl[
    h(\vec w) = v_m \And \bigl(\bigand_{j=1}^n t_j = w_j\bigr) \bigr]$;
\item[(b)] replace $\sigma$ by the sentence
\[
\textstyle
\forall \vec v \vec w\, \bigl[ \bigl(\bigl( \bigand_{i\ne  k}
\alpha_i(\vec v) \bigr)\And h(\vec w)=v_m\And
\bigl(\bigand_{j=1}^n t_j=w_j\bigr)\bigr) \to
\gamma(\vec v)\bigr].
\]
\end{enumerate}

 \item[(4)]
Else if $\gamma$ is of the form $r(t_1,\dotsc,t_n)$, with some
$t_\ell$ not a variable:
\begin{enumerate}
\item[(a)] use $\gamma \equiv \bigl(\bigand_{j=1}^n t_j =
    t_j\bigr)\And \forall \vec w\,\bigl[ \bigl(\bigand_{j=1}^n t_j =
    w_j\bigr) \to r(\vec w)\bigr]$;
\item[(b)] replace $\sigma$ by the $n+1$ sentences
\begin{align*}
&\textstyle\forall \vec v\, \bigl[ \bigl(
\bigand_i \alpha_i(\vec v) \bigr) \to
t_j=t_j\bigr], \quad\text{for $j\in\{1, \dots, n\}$, and}\\
&\textstyle\forall \vec v \vec w\,\bigl[ \bigl( \bigl( \bigand_i
\alpha_i(\vec v) \bigr) \And \bigl( \bigand_{j=1}^n t_j=w_j\bigr)\bigr) \to r(\vec w)\bigr].
\end{align*}
\end{enumerate}

\item[(5)] Else if $\gamma$ is of the form $s=t$, where $s$ and $t$
    are distinct terms, at least one of which is not a variable:
\begin{enumerate}
\item[(a)] use $\gamma \equiv \bigl(s=s\And t=t\And \forall
    ww' \bigl[ \bigl(s=w\And t=w'\bigr) \to w=w' \bigr]\bigr)$;
\item[(b)] replace $\sigma$ by the three sentences
\begin{align*}
&\textstyle\forall \vec v\,\bigl[ \bigl( \bigand_i
\alpha_i(\vec v)\bigr) \to s=s\bigr],\\
&\textstyle\forall \vec v\,\bigl[ \bigl( \bigand_i \alpha_i(\vec v)\bigr) \to t=t\bigr] ,\\
&\textstyle\forall \vec v ww'\, \bigl[ \bigl( \bigl( \bigand_i
 \alpha_i(\vec v)\bigr)\And s=w\And t=w' \bigr) \to w=w'\bigr].
\end{align*}
\end{enumerate}

\goodbreak

 \item[(6)]
Else $\gamma$ is of the form $h(t_1,\dotsc,t_n) =
h(t_1,\dotsc,t_n)$, with some $t_\ell$ not a variable:
\begin{enumerate}
\item[(a)] use $\gamma \equiv \bigl(\bigand_{j=1}^n t_j =
    t_j\bigr)\And \forall \vec w\,\bigl[ \bigl(\bigand_{j=1}^n t_j =
    w_j\bigr) \to h(\vec w) = h(\vec w)\bigr]$;
\item[(b)] replace $\sigma$ by the $n+1$ sentences
\begin{align*}
 &\textstyle\forall \vec v\, \bigl[ \bigl(
\bigand_i \alpha_i(\vec v) \bigr) \to
t_j=t_j\bigr], \quad\text{for $j\in\{1, \dots, n\}$, and}\\
 &\textstyle\forall \vec v \vec w\,\bigl[ \bigl(
 \bigl( \bigand_i \alpha_i(\vec v) \bigr)\And \bigl(
 \bigand_{j=1}^n t_j=w_j\bigr)\bigr) \to h(\vec w)=h(\vec w)\bigr].
\qedhere
\end{align*}
\end{enumerate}
\end{newlist}
\end{proof}

\begin{notation}\label{not:rel}
Given a structure $\X$ and an $n$-variable sentence $\sigma$ of the form
${\forall \vec v \,\bigl[ \psi(\vec v) \to\, \gamma(\vec v)\bigr]}$
in the language of~$\X$, we use $\rel \X \sigma$ to denote the
$n$-ary relation on $X$ defined by the premise of~$\sigma$, that is,
\[
 \rel \X \sigma := \bigl\{\, (x_1,\dotsc,x_n) \in X^n \bigm|
   \X \models \psi(x_1,\dotsc,x_n)  \,\bigr\}.
\]
\end{notation}

\begin{definition}
Let $r$ be a $k$-ary relation on $M$ and let $s$ be an $\ell$-ary relation on~$M$.
We say that $r$ is a \emph{bijective projection} of~$s$ if there is a bijection
$\rho \colon s \to r$ of the form $\rho(a_1,\dotsc,a_\ell) = (a_{\theta(1)},\dotsc,a_{\theta(k)})$,
for some map $\theta \colon \{1,\dotsc,k\} \to \{1,\dotsc,\ell\}$.
\end{definition}

\begin{remark}\label{rem:extra}
Let $\sigma$ be a uH-sentence and let $\Phi$ be the logically
equivalent set of pure uH-sentences obtained via the proof of
the previous lemma. In Section~\ref{sec:trans}, we will use the
following two facts.
\begin{enumerate}
\item For all $\phi \in \Phi$, the conclusion of $\phi$ is in the
same language as the conclusion of the original
uH-sentence~$\sigma$. That is, any partial-operation
or relation symbol occurring in the conclusion of $\varphi$ also occurs
in the conclusion of~$\sigma$.

\item For all $\phi \in \Phi$, the premise of the original
uH-sentence $\sigma$ is a `bijective projection' of the premise
of~$\phi$. That is, for each structure $\X$
such that $\X \models \sigma$, the relation $\rel \X \sigma$ is a
bijective projection of the relation $\rel \X \phi$.
\end{enumerate}
\end{remark}

\section{Proof of the New-from-old Theorem: The sharp functor}\label{sec:sharp}

The New-from-old Theorem~\ref{thm:trans} will give conditions
under which we can deduce that an alter ego $\MT_2$ fully
dualises $\MB$ [at the finite level] if we know that another
alter ego $\MT_1$ fully dualises $\MB$ [at the finite level].
In this section and the next, we set up and prove the theorem.

\begin{assumptions}\label{su:init}
Fix a finite algebra $\MB$ and define $\CA := \ISP\MB$. Let
\[
\MT_1 = \langle M; H_1, R_1, \T \rangle \quad\text{and}\quad
\MT_2 = \langle M; H_2, R_2, \T \rangle
\]
be two alter egos of $\MB$, and assume that
\begin{enumerate}
 \item[\Ahom] every hom-minimal relation on $\MB$ belongs to $\cadef {\MT_2}$, and
 \item[\Aop] $\MT_2$ is operationally rich at each relation in $R_2 \cup \dom {H_2}$.
\end{enumerate}
\end{assumptions}

To mimic the set-up for our motivating example from Section~\ref{sec:three}, take $\MB$
to be the bounded lattice $\ThB$ and choose $\MT_1 = \ThT_\sigma$ and $\MT_2 = \ThT_h$.

Note that the two conditions {\Ahom} and {\Aop} are necessary for $\MT_2$
to yield a finite-level full duality, by the Full Duality Lemma~\ref{lem:fdual}.
Using the following easy lemma, the assumption {\Aop} also ensures that $\MT_2$ is
operationally rich at each relation in~$\graph {H_2}$.

\begin{lemma}\label{lem:graph-dom}
Let\/ $\MT$ be an alter ego of a finite algebra~$\MB$. Let\/
$r$ and\/ $s$ be relations compatible with\/~$\MB$\up, and assume that
there is a bijective projection\/ $\rho \colon \sb \to \rb$.
If\/ $\MT$ is operationally rich at~$r$\up, then\/ $\MT$ is also
operationally rich at~$s$.
\end{lemma}
\begin{proof}
Say that $r$ is $m$-ary and $s$ is $n$-ary. The projection
$\rho \colon \sb \to \rb$ is given by $\rho(a_1,\dotsc,a_n) =
(a_{\theta(1)},\dotsc,a_{\theta(m)})$, for some map $\theta
\colon \{1,\dotsc,m\} \to \{1,\dotsc,n\}$.

Assume that $\MT$ is operationally rich at $r$. Let $h \colon
\sb \to \MB$ be a partial operation compatible with~$\MB$.
Then $h \circ \rho^{-1} \colon \rb \to \MB$ is also
a partial operation compatible with~$\MB$. So there is an $m$-ary term
$t(v_1,\dotsc,v_m)$ in the language of $\MT$ such that $h \circ
\rho^{-1}(a_1,\dotsc,a_m) = t^{\MT}(a_1,\dotsc,a_m)$, for all
$(a_1,\dotsc,a_m) \in r$. Define the $n$-ary term
$t_1(v_1,\dotsc,v_n) = t(v_{\theta(1)},\dotsc,v_{\theta(m)})$.
Then, for any $(a_1,\dotsc,a_n) \in s$, we have
\[
 h(a_1,\dotsc,a_n) = h \circ \rho^{-1} \bigl(\rho(a_1,\dotsc,a_n)\bigr)
  = t^{\MT}(a_{\theta(1)},\dotsc,a_{\theta(m)}) = t_1^{\MT}(a_1,\dotsc,a_n).
\]
Thus $t_1^\MT$ is an extension of $h$ in $\clo \MT$.
\end{proof}

\begin{notation}\label{not:xy}
We denote the top alter ego of~$\MB$ by
$\MT_\Omega = \langle M; H_\Omega, R_\Omega, \T\rangle$;
see Section~\ref{sec:alteregos}.
Now, for each $k \in \{1,2,\Omega\}$, let $\CZ_k$ denote
the category of all Boolean structures of signature $(H_k,R_k)$, and
define the two full subcategories
\[
\CX_k := \IScP {\MT_k}
 \quad\text{and}\quad
\CY_k := \{\, \Y \in \CZ_k \mid \Y \models \ThuH {\MT_k} \,\}
\]
within $\CZ_k$; note that $\CX_k \subseteq \CY_k$. For each
$k \in \{1,2\}$, let $F_k\colon\CZ_\Omega \to \CZ_k$ be the
natural forgetful functor.
\end{notation}

Our aim in this section is to set up a `sharp' functor $S_2
\colon \CY_2 \to \CZ_\Omega$ that enriches each Boolean model
of $\ThuH {\MT_2}$ into a Boolean structure of signature $(H_\Omega,
R_\Omega)$. This mimics our motivating example in
Section~\ref{sec:three}, where we enriched each finite
structure $\X \in \CX_h$ into a structure $\X^\sharp \in
\CZ_\sigma$ by defining the graph of the partial operation
$\psigma \X$ conjunct-atomically in the language of~$\ThT_h$.
In the general situation, not every compatible relation on $\MB$ is
conjunct-atomic definable from~$\MT_2$. But we now show that
assumption~\ref{su:init}{\Ahom} ensures that every compatible
relation on $\MB$ is primitive-positive definable from~$\MT_2$.

\begin{definition}\label{defn:betas_pt1}
For each $n$-ary compatible relation $r$ on~$\MB$, where $n \ge
0$, fix an enumeration $f_1,\dotsc,f_m$ of the hom-set $\rhom\CA(\rb,\MB)$ and
define the $(n+m)$-ary compatible relation
\[
\widehat r := \bigl\{\, (\vec a, f_1(\vec a), \dotsc, f_m(\vec a))
\bigm| \vec a \in r \,\bigr\}
\]
on $\MB$.
\end{definition}

In the definition above, the algebra $\rb \le \MB^n$ is the
isomorphic projection of the algebra~$\widehat \rb \le
\MB^{n+m}$ onto its first $n$ coordinates. By construction,
the relation $\widehat r$ is hom-minimal on~$\MB$. Therefore
$\widehat r$ is conjunct-atomic definable from~$\MT_2$, by
assumption~\ref{su:init}{\Ahom}.  This justifies the next
definition.

\begin{definition}\label{defn:betas_pt2}
For each $n$-ary compatible relation $r$ on~$\MB$, with
$\widehat r$ the associated $(n + m)$-ary hom-minimal relation
on~$\MB$,
\begin{enumerate}
\item[(a)] fix an $(n+m)$-variable conjunct-atomic formula
    $\widehat \beta_r(\vec v,\vec w)$ in the language of
    $\MT_2$ that defines $\widehat r$ in~$\MT_2$, and
\item[(b)] define the primitive-positive formula
    $\beta_r(\vec v) := \exists \vec w\, \widehat
    \beta_r(\vec v,\vec w)$.
\end{enumerate}
\end{definition}

\begin{lemma}\label{lem:entails}
Let\/ $r$ be an $n$-ary compatible relation on~$\MB$\up, for some
$n \ge 0$. Then the formula $\beta_r(\vec v)$ defines the
relation $r$ in~$\MT_2$.
\end{lemma}
\begin{proof}
For all $\vec a \in M^n$, we have the sequence of equivalences
\[
\MT_2 \models \beta_r(\vec a)
 \ \iff \ \exists \vec c \in M^m\, (\vec a,\vec c) \in \widehat r
 \ \iff \ \vec a \in r,
\]
as required.
\end{proof}

\begin{lemma}\label{2.4.5}
Let\/ $r$ be an $n$-ary compatible relation on~$\MB$\up, for
some $n \ge 0$. Let $\X \in \CY_2$ and let\/ $r_\X$ denote the
$n$-ary relation defined in $\X$ by the formula $\beta_r(\vec v)$.
\begin{enumerate}
 \item The relation $r_\X$ is topologically closed in~$\X^n$.
 \item If\/ $r$ is the graph of a partial operation
     on~$M$\up, then $r_\X$ is the graph of a continuous
     partial operation on $\X$ with a topologically closed
     domain.
\end{enumerate}
\end{lemma}
\begin{proof}
(1): Let $\widehat r_\X$ denote the $(n + m)$-ary relation
defined in $\X$ by the conjunct-atomic formula $\widehat
\beta_r(\vec v, \vec w)$. Then the relation $\widehat r_\X$ is
topologically closed in $\X^{n+m}$, since $\X$ is a Boolean structure.
But $r_\X$ is just the projection of $\widehat r_\X$ onto its first
$n$ coordinates. Since $\X^{n+m}$ is compact and $\X^n$ is Hausdorff,
it follows that $r_\X$ is also topologically closed.

(2): Let $r$ be the graph of an $n$-ary compatible partial
operation on~$\MB$, with corresponding $(n+1)$-variable
primitive-positive formula $\beta_r(\vec v,u)$. The sentence
\begin{equation*}\tag{$\dag$}
\forall \vec v \, uu'\, \bigl[ \bigl(\beta_r(\vec v,u) \And
\beta_r(\vec v,u') \bigr) \to u = u' \bigr]
\end{equation*}
is logically equivalent to a uH-sentence in the language
of~$\MT_2$. Since $\beta_r(\vec v,u)$ defines $r$ in~$\MT_2$
(by Lemma~\ref{lem:entails}),
the sentence~($\dag$) is true in $\MT_2$ and therefore true
in~$\CY_2$. Thus $r_\X$ is the graph of an $n$-ary partial
operation $h$ on $X$. It follows from part~(1) that $r_\X$ is
closed. Since the codomain of $h$ is compact and Hausdorff
and the graph of $h$ is closed, it follows that $h$ is continuous.
The domain of $h$ is closed as it is a projection of~$r_\X$
from the compact space $\X^{n+1}$ to the Hausdorff space~$\X^n$.
\end{proof}

\begin{definition}\label{def:sfunctor}
Define the sharp functor $S_2 \colon \CY_2 \to \CZ_\Omega$ as
follows.
\begin{enumerate}
 \item For each structure $\X \in \CY_2$, define $S_2(\X)$
     to be the Boolean structure of signature $(H_\Omega,
     R_\Omega)$ such that:
\begin{itemize}
 \item $S_2(\X)$ has the same underlying set and
     topology as~$\X$;
 \item for all $r \in R_\Omega$, the relation
     $r^{S_2(\X)}$ is the relation $r_\X$ defined in
     $\X$ by the formula $\beta_r(\vec v)$;
 \item for all $h \in H_\Omega$, the graph of the
     partial operation $h^{S_2(\X)}$ is the relation
     $\graph h_\X$ defined in $\X$ by the formula
     $\beta_{\graph h}(\vec v,u)$.
\end{itemize}
(Note that $S_2(\X) \in \CZ_\Omega$, by the previous
 lemma.)
 \item For each morphism $\mu \colon \X \to \Y$ in~$\CY_2$,
     the morphism $S_2(\mu) \colon S_2(\X) \to S_2(\Y)$ has
     the same underlying set-map as~$\mu$. (This works
     because morphisms are compatible with
     primitive-positive formul\ae.)
\end{enumerate}
\end{definition}

\begin{note}\label{not:sharp}
It follows at once from Lemma~\ref{lem:entails} that
$S_2(\MT_2) = \MT_\Omega$.
\end{note}

\begin{lemma}\label{lem:dom}
Let\/ $h$ be a compatible partial operation on~$\MB$\up, and
let\/ $\X \in \CY_2$. Then\/ $\dom {h^{S_2(\X)}} = \dom h^{S_2(\X)}$.
\end{lemma}
\begin{proof}
Consider the compatible relations $r := \dom h$ and $s := \graph
h$ on~$\MB$. Let the fixed enumerations used in
Definition~\ref{defn:betas_pt1} be $f_1,\dotsc,f_m$ for
$\rhom\CA(\rb,\MB)$ and $g_1,\dotsc,g_m$ for $\rhom\CA(\sb,\MB)$. Note
that the two hom-sets have the same size, since there is an
isomorphism $\rho \colon \sb \to \rb$, given by $\rho(\vec a,h(\vec
a)) := \vec a$. Indeed, there is a permutation $\theta$ of
$\{1,\dotsc,m\}$ such that $g_i = f_{\theta(i)} \circ \rho$,
for all $i \in \{1,\dotsc,m\}$.

We now have
\begin{align*}
\widehat r &= \bigl\{\, \bigl(\vec a,f_1(\vec a),\dotsc,f_m(\vec a)\bigr) \bigm|
  \vec a \in r \,\bigr\} \quad\text{and} \\
 \widehat s &= \bigl\{\, \bigl(\vec a,h(\vec a),f_{\theta(1)}(\vec a),\dotsc,f_{\theta(m)}(\vec a)\bigr)
   \bigm| \vec a \in r \,\bigr\}.
\end{align*}
We can choose $j \in \{1,\dotsc,m\}$ such that $h = f_j$. The
sentence
\[
\forall \vec v u w_1 \dotsm w_m \, \bigl[
  \bigl( \widehat \beta_r(\vec v,w_1,\dotsc,w_m) \And u = w_j \bigr)
  \leftrightarrow
  \widehat \beta_s(\vec v,u,w_{\theta(1)},\dotsc,w_{\theta(m)})
  \bigr]
\]
is equivalent to a conjunction of uH-sentences in the language
of~$\MT_2$. This sentence is true in $\MT_2$ and therefore
in~$\CY_2$, and logically implies the sentence
\[
\forall \vec v \, \bigl[
  \beta_{\dom h}(\vec v)
  \leftrightarrow
  (\exists u)\, \beta_{\graph h} (\vec v,u)
  \bigr].
\]
Hence $\dom h^{S_2(\X)} = \dom {h^{S_2(\X)}}$.
\end{proof}

\begin{lemma}\label{lem:reduct}
Let $\X \in \CY_2$. Then
\begin{enumerate}
 \item $r^{S_2(\X)} = r^\X$\up, for each $r \in R_2$\up, and
 \item $h^{S_2(\X)}  = h^\X$\up, for each $h \in H_2$.
\end{enumerate}
\end{lemma}
\begin{proof}
In cases (1) and (2), respectively, let $s$ be the compatible
relation $r$ or $\graph h$ and let $\alpha(\vec v)$ be the
atomic formula $r(v_1,\dotsc,v_n)$ or
$h(v_1,\ldots,v_{n-1})=v_n$. Then the sentence $\forall \vec v\,
\bigl[\beta_s(\vec v) \leftrightarrow \alpha(\vec v)\bigr]$ is
true in~$\MT_2$ (by Lemma~\ref{lem:entails}),
and it suffices to prove that this sentence is
true in~$\CY_2$. Since the implication $\forall \vec v\, \bigl[
\beta_s(\vec v) \to \alpha(\vec v) \bigr]$ is logically
equivalent to a uH-sentence and is true in~$\MT_2$, it is true
in~$\CY_2$. So it remains to consider the converse implication.

By assumption~\ref{su:init}{\Aop}, the alter ego $\MT_2$ is operationally rich at each
relation in $R_2 \cup \dom {H_2}$. By Lemma~\ref{lem:graph-dom}, it follows that $\MT_2$
is operationally rich at each relation in $\graph {H_2}$.  So $\MT_2$ is operationally
rich at~$s$. Let $f_1,\dotsc,f_m$ be the fixed enumeration of
$\rhom\CA(\sb,\MB)$ used in
Definition~\ref{defn:betas_pt1}. Then $f_1,\dotsc,f_m$ have extensions $g_1,\dotsc,g_m$
in $\clo {\MT_2}$ by operational richness. Thus
\[
\widehat s = \bigl\{\, \bigl(\vec a,f_1(\vec a),\dotsc,f_m(\vec
a)\bigr) \bigm| \vec a \in s\,\bigr\} = \bigl\{\, \bigl(\vec
a,g_1(\vec a),\dotsc,g_m(\vec a)\bigr) \bigm| \vec a \in s\,\bigr\}.
\]
Let $t_1,\dotsc,t_m$ be terms in the language of~$\MT_2$ that
yield the partial operations $g_1,\dotsc,g_m$. Then the
sentence
\begin{equation*}\tag{$\dag$}
\forall \vec v\, \bigl[\alpha(\vec v) \to \widehat \beta_s\bigl(\vec
v,t_1(\vec v),\dotsc,t_m(\vec v)\bigr)\bigr]
\end{equation*}
is equivalent to a conjunction of uH-sentences in the language
of~$\MT_2$. The sentence~($\dag$) holds in~$\MT_2$ and thus
in~$\CY_2$. But ($\dag$) logically implies $\forall \vec v\,
\bigl[\alpha(\vec v) \to \beta_s(\vec v)\bigr]$, as required.
\end{proof}

\begin{lemma}\label{lem:forget}\
\begin{enumerate}
 \item For each $\X \in \CY_2$\up, we have $\X = F_2 S_2 (\X)$.
 \item For each $\X \in \CY_\Omega$\up, we have $\X = S_2 F_2 (\X)$.
\end{enumerate}
\end{lemma}
\begin{proof}
Part~(1) follows directly from the previous lemma. To prove part~(2),
first note that the top alter ego $\MT_\Omega$ satisfies the
assumptions~\ref{su:init}{\Ahom} and \ref{su:init}{\Aop}
(i.e., with $\MT_\Omega$ replacing $\MT_2$).
Consider the functor $S_\Omega \colon \CY_\Omega \to \CZ_\Omega$
obtained by applying Definition~\ref{def:sfunctor} with $\MT_\Omega$ as~$\MT_2$,
but using the same formul\ae\ $\beta_r$ as for~$\MT_2$.
Then $S_\Omega = S_2  F_2$. The previous lemma with $\MT_\Omega$ as $\MT_2$
yields $S_\Omega = \id_{\CY_\Omega}$.
\end{proof}

\section{Proof of the New-from-old Theorem: The transfer functor}\label{sec:trans}

Throughout this section, the assumptions~\ref{su:init}{\Ahom}
and~\ref{su:init}{\Aop} remain in force.
In the previous section, we defined the sharp functor $S_2
\colon \CY_2 \to \CZ_\Omega$. In this section, we aim to show
that the transfer functor
\[
T_{21} := F_1  S_2 \colon \CY_2 \to \CY_1
\]
is well defined. As in the motivating example from
Section~\ref{sec:three}, we will use a basis
$\Sigma_1$ for the uH-theory of~$\MT_1$. By
Lemma~\ref{lem:pure}, we can assume that all the sentences in
$\Sigma_1$ are pure. We will need to strengthen our assumptions
on~$\MT_2$, but to do this we require some definitions.

\begin{definition}\label{defn:nat}
Let $\phi = \forall \vec v \,\bigl[ \bigl(\bigand_{i=1}^\nu
\alpha_i(\vec v) \bigr) \to \gamma(\vec v)\bigr]$ be a pure
uH-sentence in the language of~$\MT_\Omega$. Define
$\phi^\natural$ to be the sentence in the language of~$\MT_2$ constructed
from $\phi$ as follows.
\begin{enumerate}
\item
First, simultaneously make the following replacements:
\begin{enumerate}
 \item replace each $r(v_{i_1}, \dotsc, v_{i_n})$ in $\phi$
     with $\beta_r(v_{i_1}, \dotsc, v_{i_n})$;
 \item replace each $h(v_{i_1}, \dotsc, v_{i_n}) = v_{i_0}$
     in $\phi$ with $\beta_{\graph h}(v_{i_1}, \dotsc,
     v_{i_n},v_{i_0})$;
 \item if the conclusion $\gamma(\vec v)$ is $h(v_{i_1},
     \dotsc, v_{i_n}) = h(v_{i_1}, \dotsc, v_{i_n})$,
     replace it with $\beta_{\dom h}(v_{i_1}, \dotsc,
     v_{i_n})$.
\end{enumerate}
\item
Then convert the new existential quantifiers in the
premise into universal quantifiers out the front.
\end{enumerate}
Let the new sentence so constructed be
\[
\phi^\natural = \forall \vec v \vec w_1 \dots \vec w_\nu \,\Bigl[
\Bigl( \bigand_{i=1}^\nu \alpha_i^\natural(\vec v,\vec w_i) \Bigr)
\to \gamma^\natural(\vec v)\Bigr].
\]
Note that each $\alpha_i^\natural$ in the premise is of the
form $\widehat\beta_r(v_{i_1}, \dotsc, v_{i_n}, \vec w_i)$, for
some $r \in R_\Omega$, and therefore is a conjunct-atomic
formula in the language of~$\MT_2$. The conclusion
$\gamma^\natural$ is either a primitive-positive formula in the
language of~$\MT_2$ or else~$\bot$.
\end{definition}

\begin{lemma}\label{lem:nat}
Let\/ $\X \in \CY_2$ and let\/ $\phi$ be a pure uH-sentence in
the language of\/~$\MT_\Omega$.  Then\/ $S_2(\X) \models \phi$
if and only if\/ $\X \models \phi^\natural$.
\end{lemma}
\begin{proof}
This follows from the definitions of $S_2(\X)$
and~$\phi^\natural$. The only complication is
replacement~\ref{defn:nat}(1)(c).
But Lemma~\ref{lem:dom} tells us
that, for all $h \in H_\Omega$, we have
\begin{align*}
S_2(\X) \models h(\vec a) = h(\vec a)
 &\iff
\vec a \in \dom {h^{S_2(\X)}}
 \iff
\vec a \in \dom h^{S_2(\X)} \\
 &\iff
\X \models \beta_{\dom h} (\vec a).
\end{align*}
So the result holds.
\end{proof}

Recall that the notation $\rel \X \sigma$ was introduced in~\ref{not:rel}.

\begin{lemma}\label{lem:new}
Let\/ $\X \in \CY_2$ and let\/ $\phi$ be a pure uH-sentence such
that\/ $\MT_\Omega \models \phi$. Then $S_2(\X) \models \phi$
provided either
\begin{enumerate}
 \item the conclusion of\/ $\phi$ is in the language of\/~$\MT_2$\up, or
 \item $\MT_2$ is operationally rich at the relation $\rel {\MT_2} {\phi^\natural}$.
\end{enumerate}
\end{lemma}
\begin{proof}
Assume that (1) or (2) holds. We want to show that $S_2(\X)
\models \phi$. By Lemma~\ref{lem:nat}, it suffices to show that
$\X \models \phi^\natural$. Using Note~\ref{not:sharp}, we have
$S_2(\MT_2) = \MT_\Omega \models \phi$. So it follows by
Lemma~\ref{lem:nat} that $\MT_2 \models \phi^\natural$.

Say that $\phi = \forall \vec v \,\bigl[ \bigl(\bigand_{i=1}^\nu
\alpha_i(\vec v)\bigr) \to \gamma(\vec v)\bigr]$ and let
$\gamma^\natural(\vec v)$ be the conclusion of $\phi^\natural$;
see Definition~\ref{defn:nat}. Since the uH-sentence $\phi$ is
pure, we know that its conclusion $\gamma(\vec v)$ must take
one of the following four forms:
\begin{enumerate}
 \item[(a)] $r(v_{i_1},\dotsc,v_{i_n})$, for some $r \in
     R_\Omega$, in which case $\gamma^\natural(\vec v)$ is
     $\beta_r(v_{i_1},\dotsc,v_{i_n})$;
 \item[(b)] $h(v_{i_1},\dotsc,v_{i_n}) =
     h(v_{i_1},\dotsc,v_{i_n})$, for some $h \in
     H_\Omega$, in which case $\gamma^\natural(\vec v)$ is
$\beta_r(v_{i_1},\dotsc,v_{i_n})$, where $r := \dom h$;
 \item[(c)] $v_{i_1} = v_{i_2}$, in which case
     $\gamma^\natural(\vec v)$ is also $v_{i_1} = v_{i_2}$;
 \item[(d)] $\bot$, in which case $\gamma^\natural(\vec v)$ is also $\bot$.
\end{enumerate}
If $\gamma(\vec v)$ is of type~(c) or~(d), then $\phi^\natural$
is a uH-sentence true in $\MT_2$ and thus in~$\X$. So we can
now assume that $\gamma(\vec v)$ is of type~(a) or~(b).

\smallskip

\emph{Case $(1)$\up: the conclusion of\/ $\phi$ is in the
language of\/~$\MT_2$.} We can construct a uH-sentence $\psi$
in the language of $\MT_2$ from $\phi^\natural$ by changing the
conclusion $\gamma^\natural(\vec v)$ back to $\gamma(\vec v)$.
The conclusion $\gamma^\natural(\vec v)$ is $\beta_r(v_{i_1},
\dotsc, v_{i_n})$, for some $r \in R_2 \cup \dom {H_2}$. We
know that $\beta_r$ defines the interpretation of $r$ in
$\MT_2$ (by Lemma~\ref{lem:entails}) and also in~$\X$ (by
Lemmas~\ref{lem:dom} and~\ref{lem:reduct}).  Thus
$\phi^\natural \leftrightarrow \psi$ is true in both $\MT_2$
and~$\X$. Since $\MT_2 \models \phi^\natural$ and $\psi$ is a
uH-sentence, it follows that $\X \models \phi^\natural$.

\smallskip

\emph{Case $(2)$\up: $\MT_2$ is operationally rich at the
relation $\rel {\MT_2} {\phi^\natural}$.} To show that
$\X \models \phi^\natural$, it is enough to find a set $\Sigma$ of
uH-sentences in the language of $\MT_2$ such that
$\MT_2 \models \Sigma$ and $\Sigma \vdash \phi^\natural$.

The conclusion $\gamma^\natural(\vec v)$ is $\beta_r(v_{i_1},
\dotsc, v_{i_n})$, for some $r \in R_\Omega \cup \dom
{H_\Omega}$.  Define the compatible relation $p := \rel
{\MT_2} {\phi^\natural}$ on~$\MB$. Since $\MT_2 \models
\phi^\natural$ and since $\beta_r$ defines $r$ in~$\MT_2$ (by
Lemma~\ref{lem:entails}), we have
\begin{align*}
(\vec a,\vec c_1,\dotsc,\vec c_\nu) \in p
 \iff{}&
\MT_2 \models \textstyle{\bigand_{i=1}^\nu} \,
\alpha_i^\natural(\vec a,\vec c_i)
 \implies
\MT_2 \models \gamma^\natural(\vec a) \\
 \implies{}&
\MT_2 \models \beta_r(a_{i_1},\dotsc, a_{i_n})
 \implies
(a_{i_1},\dotsc,a_{i_n}) \in r.
\end{align*}
Let $f_1,\dotsc,f_m$ be the fixed enumeration of $\rhom\CA(\rb,\MB)$
used in Definition~\ref{defn:betas_pt1}. Then, for all $j \in
\{1,\dotsc,m\}$, we can define $g_j \colon \pb \to \MB$ by
\[
g_j(\vec a,\vec c_1,\dotsc,\vec c_\nu) := f_j(a_{i_1},\dotsc,a_{i_n}).
\]
Each $g_j$ is a compatible partial operation on $\MB$ with
domain~$p$. We are assuming that $\MT_2$ is operationally rich
at $\rel {\MT_2} {\phi^\natural} = p$. Thus there are terms
$t_1, \dotsc, t_m$ in the language of~$\MT_2$ that define
extensions of $g_1,\dots,g_m$ in~$\MT_2$. Define the sentence
\begin{multline*}
\psi := \forall \vec v \vec w_1 \dots \vec w_\nu \, \bigl[
 \bigl( \bigand_{i=1}^\nu \alpha_i^\natural(\vec v,\vec w_i) \bigr) \to \\
 \widehat \beta_r\bigl(v_{i_1}, \dotsc, v_{i_n},t_1(\vec v,\vec
w_1,\dotsc,\vec w_\nu),\dotsc,t_m(\vec v,\vec w_1,\dotsc,\vec
w_\nu)\bigr) \bigr].
\end{multline*}
Then $\psi$ is equivalent to a conjunction of uH-sentences in
the language of~$\MT_2$, with $\MT_2 \models \psi$ and $\psi
\vdash \phi^\natural$. Hence it follows that $\X \models
\phi^\natural$, as required.
\end{proof}

The next lemma will be used later to simplify the checking of
condition~\ref{lem:new}(2).

\begin{lemma}\label{lem:q}
Let\/ $\phi$ be a pure uH-sentence in the language of\/
$\MT_\Omega$\up, and define the sentence $\phi^\natural$ as
in~\ref{defn:nat}. If\/ $\MT_2$ is operationally rich at the
relation\/ $\rel {\MT_\Omega} \phi$\up, then\/ $\MT_2$ is also
operationally rich at\/ $\rel {\MT_2} {\phi^\natural}$.
\end{lemma}
\begin{proof}
By Lemma~\ref{lem:graph-dom}, it is enough to show that $\rel
{\MT_\Omega} \phi$ is a bijective projection of $\rel
{\MT_2} {\phi^\natural}$. Referring to the notation of
Definition~\ref{defn:nat}, first note that each
$\alpha^\natural_j(\vec v,\vec w_j)$ in the premise of $\phi^\natural$
is of the form $\widehat \beta_{r_j} (\vec v_j,\vec w_j)$,
for some compatible relation
$r_j$ on~$\MB$, some tuple $\vec v_j = (v_{i_{j,1}},\dotsc,
v_{i_{j,n_j}})$ of variables from $\vec v$, and some tuple
of new variables $\vec w_j$ of length $|\rhom\CA(\rb_j,\MB)|$. Let
$\vec f_j$ be the fixed enumeration of $\rhom\CA(\rb_j,\MB)$ used in
Definition~\ref{defn:betas_pt1}. Then
\begin{align*}
  (\vec a, \vec c_1, \dotsc, \vec c_\nu) \in \rel{\MT_2}{\phi^\natural}
 &\iff
\MT_2 \models \bigand_{j=1}^\nu \alpha_j^\natural(\vec a,\vec c_j) \\
 &\iff
\MT_2 \models \bigand_{j=1}^\nu
\widehat \beta_{r_j} (\vec a_j,\vec c_j) \\
 &\iff
  \bigl(\forall j \in \{1,\dotsc,\nu\}\bigr)\
   \bigl( \vec a_j \in r_j  \And  \vec c_j = \vec f_j(\vec a_j)
   \bigr) \\
 &\iff
  \vec a \in \rel{\MT_\Omega}{\phi} \And
  \bigl(\forall j \in \{1,\dotsc,\nu\}\bigr)\
    \vec c_j = \vec f_j(\vec a_j).
\end{align*}
It now follows that $\rho \colon \rel{\MT_2}{\phi^\natural}
\to \rel{\MT_\Omega}{\phi}$, given by $(\vec a,\vec
c_1,\ldots,\vec c_\nu) \mapsto \vec a$, is a bijective
projection.
\end{proof}

We now add to our initial assumptions,~\ref{su:init}, in order
to ensure that the transfer functor $T_{21} := F_1  S_2
\colon \CY_2 \to \CY_1$ is well defined.

\begin{assumptions}\label{su:more}
Choose a basis $\Sigma_1$ for the universal Horn theory of~$\MT_1$
such that each uH-sentence in $\Sigma_1$ is pure. Assume that
\begin{enumerate}
 \item[\Aaxi] for each $\phi \in \Sigma_1$, if the
     conclusion of $\phi$ is not in the language
     of~$\MT_2$, then $\MT_2$ is operationally rich at the
     relation $\rel {\MT_2}{\phi^\natural}$.
\end{enumerate}
\end{assumptions}

\begin{note}\label{not:axi}
Each relation $\rel {\MT_2}{\phi^\natural}$
is conjunct-atomic definable from hom-minimal relations
on~$\MB$; see Definition~\ref{defn:nat}. So
assumption~\ref{su:more}{\Aaxi} is necessary for $\MT_2$
to yield a finite-level full duality, by the Full
Duality Lemma~\ref{lem:fdual}.
\end{note}

\begin{lemma}\label{lem:transf}
The transfer functor $T_{21} := F_1  S_2 \colon \CY_2 \to \CY_1$
is well defined. That is\up, if\/ $\X \in \CY_2$\up, then\/ $F_1  S_2(\X) \in \CY_1$.
\end{lemma}
\begin{proof}
Let $\X \in \CY_2$ and $\phi \in \Sigma_1$. Then $S_2(\X) \models \phi$,
using~\ref{su:more}{\Aaxi} and Lemma~\ref{lem:new}.
So $F_1S_2(\X) \models \phi$. It follows that $F_1S_2(\X) \models
\Sigma_1$ and therefore $F_1S_2(\X) \in \CY_1$.
\end{proof}

\begin{remark}\label{rem:symm}
Suppose that, in addition to our assumptions~\ref{su:init} and~\ref{su:more}
on~$\MT_2$, we assume that $\MT_1$~fully dualises~$\MB$ at the finite level.
Then $\MT_1$ also satisfies conditions~\ref{su:init}{\Ahom} and~\ref{su:init}{\Aop}
(i.e., with $\MT_1$ replacing $\MT_2$) by the Full Duality Lemma~\ref{lem:fdual}.
This means that we can use the method of Section~\ref{sec:sharp}
to define a sharp functor $S_1 \colon \CY_1 \to \CZ_\Omega$ based on~$\MT_1$.
As in Definition~\ref{defn:betas_pt2}, for each $r \in R_\Omega$,
we will need to choose some conjunct-atomic formula $\widehat\delta_r(\vec v,\vec w)$
in the language of $\MT_1$ that defines $\widehat r$ in~$\MT_1$, and this
formula may well be different from the one $\widehat\beta_r(\vec v,\vec w)$ chosen
for~$\MT_2$. The alter ego $\MT_1$ also satisfies condition~\ref{su:more}{\Aaxi},
for any pure basis $\Sigma_2$ for~$\ThuH{\MT_2}$, by Note~\ref{not:axi}.
So we can follow the method of this section to establish that the transfer functor
$T_{12} := F_2  S_1 \colon \CY_1 \to \CY_2$ is well defined; see Lemma~\ref{lem:transf}.
\end{remark}

\begin{lemma}\label{lem:inv}
Assume that\/ $\MT_1$ fully dualises $\MB$ at the finite level.
Then the two transfer functors\/ $T_{12} := F_2  S_1
\colon \CY_1 \to \CY_2$ and\/ $T_{21} := F_1  S_2 \colon
\CY_2 \to \CY_1$ are mutually inverse category isomorphisms.
\end{lemma}
\begin{proof}
By Lemma~\ref{lem:transf} and Remark~\ref{rem:symm},
the two transfer functors are well defined.
We just need to show that they are mutually inverse.

By the symmetry between the definitions of the functors~$T_{12}$ and~$T_{21}$
(see Remark~\ref{rem:symm}), it is enough to show that $\X = T_{12}T_{21}(\X)$,
for some arbitrary $\X \in \CY_2$. Let $r \in R_2 \cup \graph
{H_2}$. We use $\widehat r$ to denote the associated
hom-minimal relation; see Definition~\ref{defn:betas_pt1}. We
have a conjunct-atomic formula $\widehat\beta_r(\vec v,\vec w)$
in the language of $\MT_2$ that defines $\widehat r$
in~$\MT_2$, and a conjunct-atomic formula
$\widehat\delta_r(\vec v,\vec w)$ in the language of $\MT_1$
that defines $\widehat r$ in~$\MT_1$; see
Definition~\ref{defn:betas_pt2} and Remark~\ref{rem:symm}. Thus
$\MT_\Omega$ satisfies the sentence $\sigma := \forall \vec v
\vec w \, \bigl[\widehat\beta_r(\vec v,\vec w) \leftrightarrow
\widehat\delta_r(\vec v, \vec w)\bigr]$.

By Lemma~\ref{lem:reduct}, the relation $r^\X$ is defined by
the formula $\exists \vec w \,\widehat\beta_r(\vec v, \vec w)$
in $S_2(\X)$. The relation $r^{T_{12}T_{21}(\X)}$ is equal to
the relation $r^{S_1F_1S_2(\X)}$, which is described by the
formula $\exists \vec w \,\widehat\delta_r(\vec v,\vec w)$ in
$S_2(\X)$. So we can show that $r^\X = r^{T_{12}T_{21}(\X)}$ by
checking that $S_2(\X)$ satisfies the sentence~$\sigma$.

First consider the backwards implication $\sigma_b := \forall
\vec v \vec w\,\bigl[\widehat\delta_r(\vec v, \vec w) \rightarrow
\widehat\beta_r(\vec v,\vec w)\bigr]$. This is logically equivalent
to a set $\Sigma_b$ of uH-sentences, each of which holds
in~$\MT_\Omega$. Using Lemma~\ref{lem:pure}, we can convert
$\Sigma_b$ into a logically equivalent set $\Phi_b$ of pure
uH-sentences. The conclusion of each sentence in $\Phi_b$ is in
the language of~$\MT_2$; see Remark~\ref{rem:extra}(1). So
$S_2(\X) \models \sigma_b$, by Lemma~\ref{lem:new}.

Now consider the forwards implication $\sigma_f := \forall \vec
v \vec w\, \bigl[\widehat\beta_r(\vec v,\vec w) \rightarrow
\widehat\delta_r(\vec v, \vec w)\bigr]$. This is logically
equivalent to a set $\Sigma_f$ of uH-sentences, each of
which holds in~$\MT_\Omega$. Using Lemma~\ref{lem:pure} again,
we can convert $\Sigma_f$ into a logically equivalent set
$\Phi_f$ of pure uH-sentences. Let $\phi \in \Phi_f$.
By Remark~\ref{rem:extra}(2), since $\MT_\Omega \models \sigma_f$,
there is a bijective projection
$\rho \colon \rel {\MT_\Omega} \phi \to \rel {\MT_\Omega}
{\sigma_f}$. The alter ego $\MT_2$ is operationally rich at the relation
$\widehat r = \rel {\MT_\Omega} {\sigma_f}$, since it is hom-minimal, and
therefore $\MT_2$ is also operationally rich at $\rel
{\MT_\Omega} \phi$, by Lemma~\ref{lem:graph-dom}. So
$S_2(\X) \models \phi$, by Lemmas~\ref{lem:new} and~\ref{lem:q}.
It follows that $S_2(\X) \models \sigma_f$, as required.
\end{proof}

We wrap up this section with the following result.

\begin{lemma}\label{lem:wildtrans}
Assume that\/ $\MT_2$ satisfies~\up{\ref{su:init}\Ahom,}
\up{\ref{su:init}\Aop} and \up{\ref{su:more}\Aaxi}.
If\/ $\MT_1$ fully dualises $\MB$ \up[at the finite
level\up]\up, then the following are equivalent\up:
\begin{enumerate}
 \item $\MT_2$ fully dualises $\MB$ \up[at the finite
     level\up]\up;
 \item the transfer functor $T_{21} := F_1  S_2 \colon
     \CY_2 \to \CY_1$ sends each \up[finite\up] structure
     in\/ $\CX_2$ into~$\CX_1$.
\end{enumerate}
\end{lemma}
\begin{proof}
(2)\,$\Rightarrow$\,(1): Assume that (2) holds.
The alter ego $\MT_2$ dualises $\MB$ at the finite level, by~\ref{su:init}{\Ahom}
and the Duality Lemma~\ref{lem:dual}. If $\MT_1$ dualises $\MB$,
then so does $\MT_2$; see Remark~\ref{rem:lift}.
By the previous lemma, the transfer functor $T_{21} \colon \CY_2 \to \CY_1$
is a category isomorphism that preserves underlying sets and set-maps,
and by Note~\ref{not:sharp} we have $T_{21}(\MT_2) = \MT_1$.
It follows that $\MT_2$ fully dualises $\MB$ [at the finite level],
using condition~(2) and the New-from-old Lemma~\ref{lem:transfer}.

(1)\,$\Rightarrow$\,(2): Assume $\MT_2$ fully dualises $\MB$
[at the finite level]. For $i\in \{1,2,\Omega\}$, let
$D_i \colon \CA \to \CX_i$ and $E_i \colon \CX_i \to \CA$
denote the hom-functors induced by $\MB$ and~$\MT_i$.
Let $\X$ be a [finite] structure in~$\CX_2$,
and define the algebra $\A := E_2(\X) \in \CA$.
Then $\X \cong D_2E_2(\X) = D_2(\A)$. Since $D_i(\A) = F_i
D_\Omega (\A)$, for each $i \in \{1,2\}$, using Lemma~\ref{lem:forget}(2) yields
\[
D_1(\A) = F_1 D_\Omega (\A)= F_1S_2F_2D_\Omega (\A)=F_1S_2D_2(\A).
\]
So we have $F_1S_2(\X) \cong F_1S_2D_2(\A) = D_1(\A) \in \CX_1$, as
required.
\end{proof}

\section{The New-from-old Theorem and its applications}\label{sec:thm}

We now have all the ingredients necessary to state and prove
our main theorem. Recall that an alter ego $\MT$ is
\emph{standard} if the potential dual category $\CX = \IScP\MT$
consists precisely of all Boolean models of $\ThuH \MT$;
see Definition~\ref{def:std}.

\begin{theorem:transfer}\label{thm:trans}
Let\/ $\MB$ be a finite algebra\up, and let\/ $\MT_1$ and\/ $\MT_2$ be alter egos
of\/~$\MB$. Assume that\/ $\MT_2$ satisfies~\up{\ref{su:init}\Ahom,}
\up{\ref{su:init}\Aop} and \up{\ref{su:more}\Aaxi}.
\begin{enumerate}
 \item If\/ $\MT_1$ fully dualises $\MB$ at the finite
     level\up, then so does\/~$\MT_2$.
 \item If\/ $\MT_1$ is standard and fully
     dualises\/~$\MB$\up, then the same is true
     of\/~$\MT_2$.
\end{enumerate}
\end{theorem:transfer}
\begin{proof}
Part~(1) follows directly from Lemma~\ref{lem:wildtrans},
because we automatically have $(\CX_1)_{\mathrm{fin}} =
(\CY_1)_{\mathrm{fin}}$; see Notation~\ref{not:xy}
and Definition~\ref{def:std}.

To prove part~(2), assume that $\MT_1$ is standard and fully
dualises~$\MB$. Since $\MT_1$ is standard, we have $\CX_1 =
\CY_1$. It follows by Lemma~\ref{lem:wildtrans} that $\MT_2$
fully dualises~$\MB$. To see that $\MT_2$ is also standard, let
$\X \in \CY_2$. Then $T_{21}(\X) \in \CY_1 = \CX_1$. As we have
shown that $\MT_2$ fully dualises~$\MB$, we can use
Lemma~\ref{lem:wildtrans} (with the subscripts $1$ and $2$ swapped)
to deduce that $T_{12}T_{21}(\X) \in \CX_2$. Therefore Lemma~\ref{lem:inv}
gives $\X = T_{12}T_{21}(\X) \in \CX_2$. Thus $\MT_2$ is standard.
\end{proof}

\begin{warning}
In the signature of the alter ego~$\MT_1 = \langle M; H_1, R_1, \T \rangle$,
all operations are considered as partial operations.  This means
we are not privileging operations that happen to be total with the logical
status of being total operations.
To apply the New-from-old Theorem~\ref{thm:trans}, the uH-basis chosen
for~$\MT_1$ must imply \emph{all} uH-sentences true in~$\MT_1$,
including those of the form
$\forall v_1\dots v_n\,\bigl[f(v_1,\dots,v_n)=f(v_1,\dots,v_n)\bigr]$,
where $f$ is an $n$-ary total operation on~$M$ for some $n\ge 0$.
\end{warning}

We now use this rather technical theorem to obtain a series of
self-contained corollaries. First we use the theorem to give a new
and very natural condition under which every finite-level full duality
lifts to the infinite level.

\begin{theorem}\label{cor:std}
Let\/ $\MB$ be a finite algebra. Assume that\/ $\MB$ is fully
dualised by a standard alter ego. If an alter ego\/ $\MT$ fully
dualises $\MB$ at the finite level\up, then\/ $\MT$ is standard
and fully dualises\/~$\MB$.
\end{theorem}
\begin{proof}
Let $\MT_1$ be a standard alter ego that fully dualises~$\MB$.
Assume that $\MT$ fully dualises $\MB$ at the finite level.
Then we can take $\MT_2 := \MT$ and the assumptions of the
New-from-old Theorem~\ref{thm:trans} are satisfied, by the
Full Duality Lemma~\ref{lem:fdual} and Note~\ref{not:axi}.
Thus $\MT$ is standard and fully dualises $\MB$.
\end{proof}

\begin{example}\label{ex:qp}
The previous theorem can be applied to
\emph{quasi-primal algebras}, that is, to finite algebras~$\MB$
such that the \emph{ternary discriminator} $t\colon M^3 \to M$
is a term function of~$\MB$, where
\[
t(x, y, z) =
 \begin{cases}
 x &\text{if $x\ne y$,}\\
 z &\text{if $x = y$.}
\end{cases}
\]
Davey and Werner~\cite[2.7]{DW:old} have shown that every
quasi-primal algebra has a standard, strongly dualising alter
ego. So, for any quasi-primal algebra~$\MB$, the finite-level
full dualities always lift to full dualities.
\end{example}

We can also use the New-from-old Theorem to refine the intrinsic
description of finite-level full dualities given by the
Full Duality Lemma~\ref{lem:fdual}.

\begin{theorem}\label{thm:fatfl_char}
Let\/ $\MT = \langle M; H, R, \T \rangle$ be an alter ego of a
finite algebra\/~$\MB$. Then the following are equivalent\up:
\begin{enumerate}
 \item $\MT$ fully dualises $\MB$ at the finite level\up;
 \item
\begin{enumerate}
 \item every hom-minimal relation on\/ $\MB$ belongs to $\cadef \MT$\up, and
 \item $\MT$ is operationally rich at each relation in\/ $\cadef \MT$\up;
\end{enumerate}
 \item
\begin{enumerate}
 \item every hom-minimal relation on\/ $\MB$ belongs to $\cadef \MT$\up,
 \item $\MT$ is operationally rich at each relation in\/ $R \cup \dom H$\up, and
 \item $\MT$ is operationally rich at each relation
     that is  conjunct-atomic definable from
     hom-minimal relations.
\end{enumerate}
\end{enumerate}
\end{theorem}
\begin{proof}
Using the Full Duality Lemma~\ref{lem:fdual},
we only need to prove that (3)\,$\Rightarrow$\,(1).
So assume that (3) holds. We check that we can apply the New-from-old
Theorem~\ref{thm:trans} with $\MT_1 = \MT_\Omega$ and $\MT_2 = \MT$.
First note that the top alter ego $\MT_\Omega$ must fully dualise $\MB$
at the finite level, by the Full Duality Lemma~\ref{lem:fdual}.
Conditions~\ref{su:init}{\Ahom} and~\ref{su:init}{\Aop} correspond to assumptions~(3)(a)
and~(3)(b). Condition~\ref{su:more}{\Aaxi} holds by assumption~(3)(c), because each relation
$\rel {\MT} {\phi^\natural}$ is conjunct-atomic definable from hom-minimal relations; see
Note~\ref{not:axi}.
\end{proof}

From the previous result, we easily obtain a `constructive'
description of the smallest full-at-the-finite-level
alter ego~$\MT_\alpha$; see Facts~\ref{fac:galois}(2).

\begin{theorem}\label{cor:alpha}
Let\/ $\MB$ be a finite algebra. Define the sets
\begin{itemize}
 \item $R_\alpha$ of all compatible relations on $\MB$ that
     are conjunct-atomic definable from the hom-minimal
     relations on\/~$\MB$\up, and
 \item $H_\alpha$ of all compatible partial operations on
     $\MB$ with domain in~$R_\alpha$.
\end{itemize}
Then $\MT_\alpha = \langle M; H_\alpha, R_\alpha,\T\rangle$
is the smallest alter ego that fully dualises $\MB$ at the finite level.
\end{theorem}
\begin{proof}
This follows from Theorem~\ref{thm:fatfl_char} (1)\,$\Leftrightarrow$\,(3).
\end{proof}

Using the transfer set-up from Section~\ref{sec:trans},
we can give a new proof of the known characterisation of when a full duality
is preserved under enriching the alter ego.

\begin{theorem}[{\cite[5.3]{DPW:galois}}]\label{thm:enrich}
Let\/ $\MT_1 = \langle M; H_1, R_1, \T \rangle$
and\/ $\MT_2 = \langle M; H_2, R_2, \T \rangle$
be alter egos of a finite algebra\/~$\MB$\up,
with\/ $\MT_1$ a structural reduct of\/~$\MT_2$.
Assume that\/ $\MT_1$ fully dualises\/~$\MB$ \up[at the finite level\/\up].
Then the following are equivalent\up:
\begin{enumerate}
\item $\MT_2$ fully dualises\/ $\MB$ \up[at the finite level\/\up]\up;
\item $\MT_2$ is operationally rich at each relation in\/
  $(R_2 {\setminus} R_1) \cup \dom {H_2 {\setminus} H_1}$.
\end{enumerate}
\end{theorem}
\begin{proof}
By the Full Duality Lemma~\ref{lem:fdual}, it suffices to prove (2)\,$\Rightarrow$\,(1).
Assume that~(2) holds. Without loss of generality, we can assume $\MT_1$ is a reduct of~$\MT_2$.

Every hom-minimal relation on $\MB$ belongs to $\cadef {\MT_1} \subseteq \cadef {\MT_2}$,
by the Duality Lemma~\ref{lem:dual}. Thus~\ref{su:init}{\Ahom} holds.
Now let $r\in R_2 \cup \dom {H_2}$. Using the Full Duality Lemma~\ref{lem:fdual},
if $r\in R_1 \cup \dom {H_1}$, then $\MT_1$ is operationally rich at~$r$,
and so $\MT_2$ is too. Otherwise, condition~(2) ensures that $\MT_2$ is operationally
rich at~$r$. Thus~\ref{su:init}{\Aop} holds.
Since the language of $\MT_1$ is contained in that of~$\MT_2$, it follows
immediately that~\ref{su:more}{\Aaxi} holds.

We now apply Lemma~\ref{lem:wildtrans} to show that
$\MT_2$ fully dualises $\MB$ [at the finite level].
Since $\MT_1$ is a reduct of~$\MT_2$, the transfer functor
$T_{21} := F_1  S_2 \colon \CY_2 \to \CY_1$
is the forgetful functor, by Lemma~\ref{lem:forget}(1).
It follows that $T_{21}$ sends each structure in~$\CX_2$ into~$\CX_1$, as required.
\end{proof}

We will now illustrate the general New-from-old Theorem using an important
example from natural duality theory: the first known
full-but-not-strong duality.

\begin{example}\label{ex:r1}
Define the four-element lattice-based algebra
\[
\QB := \langle \{0,a,b,1\}; t, \vee, \wedge, 0, 1\rangle,
\]
where $0 < a < b < 1$ and the operation $t$ is the ternary
discriminator. Define two alter egos of $\QB$:
\[
\QT_0 := \langle \{0,a,b,1\}; \graph f,\T\rangle
 \quad\text{and}\quad
\QT_1 := \langle \{0,a,b,1\}; f,g,\T\rangle,
\]
where the partial automorphisms $f$ and $g$ of~$\QB$ are shown
in Figure~\ref{fig:fg}.

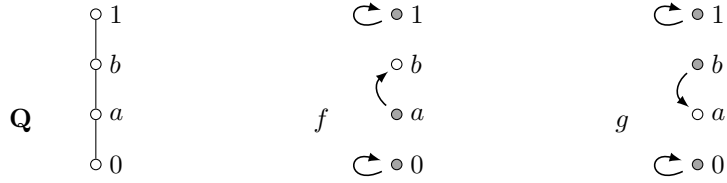
\begin{figure}[t]
\begin{tikzpicture}
   \begin{scope}
      \node[anchor=base] at (-1,0.5) {$\QB$};
      \node[unshaded] (0) at (0,0) {}; \node[label] at (0) [right=5pt] {$0$};
      \node[unshaded] (a) at (0,0.667) {}; \node[label] at (a) [right=5pt] {$a$};
      \node[unshaded] (b) at (0,1.333) {}; \node[label] at (b) [right=5pt] {$b$};
      \node[unshaded] (1) at (0,2) {}; \node[label] at (1) [right=5pt] {$1$};
      \draw[order] (0) to (a); \draw[order] (a) to (b); \draw[order] (b) to (1);
   \end{scope}
   \begin{scope}[xshift=4cm]
      \node[anchor=base] at (-1,0.5) {$f$};
      \node[shaded] (0) at (0,0) {}; \node[label] at (0) [right=5pt] {$0$};
      \node[shaded] (a) at (0,0.667) {}; \node[label] at (a) [right=5pt] {$a$};
      \node[unshaded] (b) at (0,1.333) {}; \node[label] at (b) [right=5pt] {$b$};
      \node[shaded] (1) at (0,2) {}; \node[label] at (1) [right=5pt] {$1$};
      \draw[loopy] (0) to [out=-155, in=-195] (0);
      \draw[short] (a) to [bend left] (b);
      \draw[loopy] (1) to [out=-155, in=-195] (1);
   \end{scope}
   \begin{scope}[xshift=8cm]
      \node[anchor=base] at (-1,0.5) {$g$};
      \node[shaded] (0) at (0,0) {}; \node[label] at (0) [right=5pt] {$0$};
      \node[unshaded] (a) at (0,0.667) {}; \node[label] at (a) [right=5pt] {$a$};
      \node[shaded] (b) at (0,1.333) {}; \node[label] at (b) [right=5pt] {$b$};
      \node[shaded] (1) at (0,2) {}; \node[label] at (1) [right=5pt] {$1$};
      \draw[loopy] (0) to [out=-155, in=-195] (0);
      \draw[short] (b) to [bend right] (a);
      \draw[loopy] (1) to [out=-155, in=-195] (1);
   \end{scope}
\end{tikzpicture}
\caption{The partial automorphisms $f$ and $g$ of $\QB$}\label{fig:fg}
\end{figure}

By the Quasi-primal Strong Duality
Theorem~\cite[3.3.13]{CD:book}, the alter ego $\QT_1$
strongly dualises~$\QB$. Since $\QT_0$ and $\QT_1$ are
clearly not structurally equivalent, the alter ego $\QT_0$
cannot strongly dualise~$\QB$. Nevertheless, the alter ego
$\QT_0$ fully dualises~$\QB$: Clark, Davey and
Willard~\cite{CDW:!} gave three different proofs to celebrate
this discovery; we will now give yet another proof.

Since we know that $\QT_1$ dualises~$\QB$, it follows easily
that $\QT_0$ dualises~$\QB$. Thus $\QT_0$
satisfies~\ref{su:init}{\Ahom}, by the Duality
Lemma~\ref{lem:dual}. Since $\graph f$ is hom-minimal on~$\QB$,
it is trivial that $\QT_0$ satisfies~\ref{su:init}{\Aop}.

As mentioned in Example~\ref{ex:qp}, every quasi-primal algebra
is strongly dualised by a standard alter ego. So $\QT_1$ is standard,
by Theorem~\ref{cor:std}. It is easy to check that the following
three uH-sentences form a basis for $\ThuH {\QT_1}$:
\begin{enumerate}
 \item $\forall uv\, \bigl[f(u) = v \to g(v) = u\bigr]$;
 \item $\forall uv\, \bigl[ g(u) = v \to f(v) = u\bigr]$;
 \item $\forall uvw \, \bigl[\bigl(f(u) = v \And f(v) = w \bigr) \to u = v\bigr]$.
\end{enumerate}
Sentence (3) is pure, but sentences (1) and (2) are not.
Sentence (1) converts into two pure uH-sentences:
\begin{enumerate}
 \item[(1a)] $\forall uv\, \bigl[f(u) = v \to g(v) = g(v)\bigr]$;
 \item[(1b)] $\forall uvw\, \bigl[(f(u) = v \And g(v) = w) \to w = u\bigr]$.
\end{enumerate}
The purification of~(2) is the pair of sentences~(2a) and~(2b) obtained from~(1a) and~(1b)
by interchanging $f$ and~$g$. Of the five sentences~(1a), (1b), (2a), (2b) and~(3),
only~(1a) and~(2a) have conclusions not in the language of~$\QT_0$.
Since $\rel {\MT_\Omega} {\mathrm{1a}} = \graph f$
and $\rel {\MT_\Omega} {\mathrm{2a}} = \graph g$,
both of which are hom-minimal, it follows from Lemma~\ref{lem:q}
that $\QT_0$ satisfies~\ref{su:more}{\Aaxi} with respect to these five sentences.
Thus $\QT_0$ is standard and fully dualises~$\QB$, by the New-from-old Theorem~\ref{thm:trans}.
\end{example}

To use the New-from-old Theorem directly, we need first to have
come up with a candidate alter ego $\MT_2$ that is going to
fully dualise $\MB$ [at the finite level]. But we can easily
adapt the New-from-old Theorem into an algorithm that can help
you to find, for your favourite finite algebra~$\MB$,
an alter ego of~$\MB$ that is equivalent to the smallest
full-at-the-finite-level alter ego~$\MT_\alpha$.

\begin{algorithm}\label{cor:alg}
Let $\MB$ be a finite algebra. You need the following:
\begin{enumerate}
 \item[(i)] An alter ego $\MT_0 = \langle M; H_0, R_0, \T\rangle$
 of~$\MB$ such that
 \begin{enumerate}
 \item[(a)] $\MT_0$ dualises $\MB$ at the finite level,
 \item[(b)] $\MT_0$ is operationally rich at each relation in $R_0 \cup\dom {H_0}$, and
 \item[(c)] $\MT_0$ is a reduct of $\MT_\alpha$.
 \end{enumerate}
(The easiest way to guarantee that~(c) holds is to ensure
that the signature of $\MT_0$ includes only total operations and
hom-minimal relations.)
 \item[(ii)] An alter ego $\MT_1$ that fully dualises $\MB$ at the finite level.
 \item[(iii)] A finite basis $\Sigma_1$ for the uH-theory of~$\MT_1$.
\end{enumerate}
Start with $\MT_2 := \MT_0$. Then an alter ego equivalent to $\MT_\alpha$
can be obtained by adding partial operations to the signature of~$\MT_2$
as follows.

For each uH-sentence $\psi \in \Sigma_1$ whose conclusion is not in the language of~$\MT_0$,
complete the following steps:
\begin{enumerate}
 \item Convert $\psi$ into a set of pure uH-sentences~$\phi_1,\dotsc,\phi_n$.
 \item For each $i \in \{1,\dotsc,n\}$ such that the conclusion of~$\phi_i$
     is not in the language of~$\MT_0$, calculate the relation $r_i$ on~$M$ as follows:
 \begin{enumerate}
 \item if the premise of $\phi_i$ is in the language of~$\MT_0$,
      then $r_i := \rel {\MT_0}{\phi_i}$;
 \item if the premise of $\phi_i$ is not in the language of~$\MT_0$,
      then $r_i := \rel {\MT_0}{\phi_i^\natural}$.
 \end{enumerate}
 (See~\ref{defn:nat} and~\ref{lem:q}.)
 \item For each relation $r_i$ calculated in step~(2),
 add all the compatible partial operations
 on~$\MB$ with domain~$r_i$ to the signature of~$\MT_2$.
\end{enumerate}
At the end of this process, you will have $\MT_2 \equiv \MT_\alpha$.
\end{algorithm}

We finish this section by demonstrating this algorithm on the
bounded lattice~$\ThB$, whereby we shall `rediscover' the
partial operation~$h$ used in Section~\ref{sec:three}.

\begin{example}\label{ex:three}
Consider the bounded lattice $\ThB = \langle
\{0,a,1\}; \vee,\wedge,0,1\rangle$, and define the two alter egos
\[
\ThT_0 := \langle \{0,a,1\}; f,g,\T \rangle
 \quad\text{and}\quad
\ThT_1 := \langle \{0,a,1\}; f,g,\sigma,\T \rangle,
\]
as in Section~\ref{sec:three}. Then $\ThT_0$ and $\ThT_1$
dualise and strongly dualise~$\ThB$, respectively.

A uH-axiomatisation for $\ThT_1$ is given in the proof of
Lemma~\ref{lem:sharp1}:
\begin{enumerate}
 \item $\forall v\,\bigl[f(v) = f(f(v)) = g(f(v)) \And g(v) =
     f(g(v)) = g(g(v))\bigr]$;
 \item $\forall uvw\,\bigl[\bigl(f(w) = u \And  g(w) =
     v\bigr) \leftrightarrow \sigma(u,v) = w\bigr]$;
 \item $\forall uv\,\bigl[\bigl(\sigma(u,v) = \sigma(u,v) \And
     \sigma(v,u) = \sigma(v,u) \bigr)\to u = v\bigr]$;
 \item $\forall uvw\,\bigl[\bigl(\sigma(u,v) = \sigma(u,v) \And
     \sigma(v,w) = \sigma(v,w) \bigr) \to \sigma(u,w) = \sigma(u,w)\bigr]$.
\end{enumerate}
We only need to consider~(4) and the forward direction of~(2).

The forward direction of~(2) converts into a pair of pure uH-sentences,
of which we need only consider~(2a):
\begin{enumerate}
\item [(2a)] $\forall uvw\,\bigl[\bigl(f(w) = u \And  g(w) = v \bigr) \to
     \sigma(u,v) = \sigma(u,v)\bigr]$;

\item [(2b)] $\forall uvwx\,\bigl[\bigl(f(w) = u \And  g(w) = v \And \sigma(u,v) = x\bigr) \to
     x = w\bigr]$.
\end{enumerate}
The premise of~(2a) defines the ternary relation
$\graph \sigma = \{000, 01a, 111\}$. This relation is hom-minimal on~$\ThB$,
so every compatible partial operation with domain $\graph \sigma$ already
has an extension in~$\clo {\ThT_0}$.

Using~(2), we can rewrite~(4) as
\begin{equation}
\begin{aligned}
 \forall uvwxy\,\bigl[\bigl(f(x) = u \And g(x) = v \And f(y) = v \And &g(y) = w\bigr)\to\\
 &\sigma(u,w) = \sigma(u,w)\bigr].
\end{aligned}
 \tag*{(4)$'$}
\end{equation}
Note that the premise of (4)$'$ is in the language of $\ThT_0$,
so step~(2)(a) of Algorithm~\ref{cor:alg} applies. The premise of~(4)$'$
defines the $5$-ary relation
\[
r := \rel {\MT_0}{(4)'} = \{ 00000, 0010a, 011a1, 11111 \}.
\]
This relation forms a four-element chain, and so there are six
homomorphisms from $\rb$ to~$\ThB$. Thus there is only one compatible
partial operation on~$\ThB$ with domain~$\rb$ that is not the restriction of a
projection. We could just add this $5$-ary partial operation to the
signature of~$\ThT_0$, and we would be done.

But instead, we note from the premise of~(4)$'$ that the
$5$-ary relation $r$ is isomorphic (via a projection) to the
binary relation defined by $g(x) = f(y)$, which is $\dom h = \{ 00, 0a, a1, 11 \}$.
The missing partial operation with domain $\rb$
is a restriction of $h(\pi_4,\pi_5)$. Since $\dom h$ is in $\cadef {\ThT_0}$
and since every compatible partial operation with domain $\dom h$ is generated
from the projections by $f,g,h$, we can add $h$ to $\ThT_0$ to obtain the familiar
alter ego~$\ThT_2 := \langle \{0,a,1\}; f,g,h, \T \rangle \equiv \ThT_\alpha$.
\end{example}

\section{Distinguishing full dualities}\label{sec:end}

In this final section, we clarify the precise sense in which there can be
two `different' full dualities based on the same algebra~$\MB$.
We first recall the categorical description of structural
equivalence; see Davey, Haviar and Willard~\cite[p.~404]{DHW:struct}.

\begin{lemma}\label{lem:se-functor}
Let\/ $\MT_1$ and\/ $\MT_2$ be alter egos of a finite
algebra\/~$\MB$. For $i \in \{1,2\}$\up, define $\CX_i :=
\IScP{\MT_i}$. Then the following are equivalent\up:
\begin{enumerate}
 \item $\MT_1$ and\/ $\MT_2$ are structurally
     equivalent\up;
 \item there is a concrete category isomorphism\/
     $F \colon \CX_2 \to \CX_1$ such that
\begin{enumerate}
 \item $F(\MT_2) = \MT_1$\up, and
 \item both $F$ and $F^{-1}$ preserve structural embeddings.
\end{enumerate}
\end{enumerate}
Moreover\up, we can take $F$ to be the natural `forgetful'
functor.
\end{lemma}

Using our transfer set-up from Section~\ref{sec:trans}, we obtain the following similar result.
As mentioned in the introduction, if two alter egos $\MT_1$ and $\MT_2$ both fully dualise~$\MB$,
then the categories $\IScP{\MT_1}$ and $\IScP{\MT_2}$ are equivalent, as they are both
dually equivalent to~$\CA = \ISP{\MB}$. We now show that these two categories are in fact
concretely isomorphic.

\begin{lemma}\label{lem:transfer2}
Let\/ $\MT_1$ and\/ $\MT_2$ be alter egos of a finite
algebra\/~$\MB$. For $i \in \{1,2\}$\up, define $\CX_i :=
\IScP{\MT_i}$. Assume that\/ $\MT_1$ fully dualises~$\MB$.
Then the following are equivalent\up:
\begin{enumerate}
 \item $\MT_2$ fully dualises $\MB$\up;
 \item there is a concrete category isomorphism\/
     $F \colon \CX_2 \to \CX_1$ such that
\begin{enumerate}
 \item $F(\MT_2) = \MT_1$\up, and
 \item $F$ preserves structural embeddings of the form $\X
     \stackrel{\mathrm{incl}}{\hookrightarrow}
     (\MT_2)^S$\up, where $X$ is closed under all
     compatible partial operations on~$\MB$.
\end{enumerate}
\end{enumerate}
Moreover\up, if\/ $\MT_1$ is a structural reduct of\/~$\MT_2$\up,
then we can take $F$ to be the natural `forgetful' functor\up, and
if\/ $\MT_2$ is a structural reduct of\/~$\MT_1$\up, then we
can take $F^{-1}$ to be the natural `forgetful' functor.
\end{lemma}
\begin{proof}
(2)\,$\Rightarrow$\,(1): Assume that we have
$F\colon \CX _2 \to \CX _1$ as in~(2).
Let $\CA := \ISP\MB$ and, for $i\in \{1,2\}$, let $D_i \colon \CA \to \CX_i$
and $E_i \colon \CX_i \to \CA$ denote the hom-functors
induced by~$\MB$ and~$\MT_i$.

We first show that $\MT_2$ dualises~$\MB$. Let $\A \in \CA$.
The functor $F$ preserves the structural embedding
$D_2(\A) \stackrel{\mathrm{incl}}{\hookrightarrow}(\MT_2)^A$, by~(2)(b),
and so $FD_2(\A)$ is an induced substructure of $F((\MT_2)^A)$.
We have $F(\MT_2)=\MT_1$, by~(2)(a). Therefore
\[
F((\MT_2)^A)=(F(\MT_2))^A=(\MT_1)^A,
\]
as the concrete category isomorphism $F\colon \CX_2 \to \CX_1$
preserves concrete products. The underlying set of $FD_2(\A)$ is
$\rhom\CA(\A,\MB)$, as $F$ also preserves underlying sets.
Hence $FD_2(\A)=D_1(\A)$, and therefore
\[
\rhom{\CX_2}(D_2(\A),\MT_2)=\rhom{\CX_1}(FD_2(\A),F(\MT_2))=\rhom{\CX_1}(D_1(\A),\MT_1),
\]
and so $E_2D_2(\A)=E_1D_1(\A)$. Since $\MT_1$ dualises~$\MB$,
it follows that $\MT_2$ does too.

For each structure $\X\in \CX_2$,
we have $\rhom{\CX_2}(\X,\MT_2)=\rhom{\CX_1}(F(\X),\MT_1)$. Hence $\MT_2$ fully
dualises~$\MB$, by the New-from-old Lemma~\ref{lem:transfer}.

(1)\,$\Rightarrow$\,(2): Assume that $\MT_2$ fully dualises~$\MB$.
Since $\MT_1$ also fully dualises~$\MB$, the transfer functors
$T_{21}\colon \CX_2 \to \CX_1$ and $T_{12}\colon \CX_1 \to \CX_2$ are well defined,
using Lemma~\ref{lem:wildtrans} twice. Thus $T_{21}\colon \CX_2 \to \CX_1$
is a concrete category isomorphism, by Lemma~\ref{lem:inv}.
We have $T_{21}(\MT_2) = \MT_1$, by Note~\ref{not:sharp},
and so $T_{21}$ satisfies~(2)(a).

Now let $\X$ be an induced substructure of $(\MT_2)^S$ with the property
that $X$ is closed under all
compatible partial operations on~$\MB$. Then $\X = F_2(\X^\sharp)$,
where $\X^\sharp \le (\MT_\Omega)^S$. Using Lemma~\ref{lem:forget}(2),
we have
\[
 T_{21}(\X) = T_{21} F_2(\X^\sharp) = F_1 S_2 F_2(\X^\sharp)
  = F_1(\X^\sharp) \le (\MT_1)^S.
\]
Note that $T_{21}((\MT_2)^S) = (T_{21}(\MT_2))^S = (\MT_1)^S$.
Hence $T_{21}$ satisfies~(2)(b).

If $\MT_1$ is a structural reduct of~$\MT_2$,
then the transfer functor $T_{21}$ is the natural `forgetful' functor,
by Lemma~\ref{lem:forget}(1).
Similarly, if $\MT_2$ is a structural reduct of~$\MT_1$, then
the inverse transfer functor $T_{12}$ is the natural `forgetful' functor.
\end{proof}

\begin{remark}\label{rem:embed}
We now demonstrate that
the notion of `structural embedding' we are using is not
always categorically expressible in the
concrete dual category arising from a full duality.

Our example is based on the quasi-primal algebra~$\QB$ from
Example~\ref{ex:r1}. We know that the two alter egos
\[
\QT_0 = \langle \{0,a,b,1\}; \graph f, \T \rangle
 \quad\text{and}\quad
\QT_1 = \langle \{0,a,b,1\}; f,g, \T \rangle
\]
fully dualise~$\QB$. For each $i \in \{0,1\}$, define
the dual category $\CX_i := \IScP{\QT_i}$.
By Lemma~\ref{lem:transfer2},
the natural `forgetful' functor $F \colon \CX_1 \to
\CX_0$ is a category isomorphism. The inverse category
isomorphism $F^{-1} \colon \CX_0 \to \CX_1$
preserves underlying sets and set-maps, but does not preserve
structural embeddings.

For example, consider the induced substructure~$\X$
of~$\QT_0$ with $X := \{ a \}$. The inclusion $i \colon \X \to \QT_0$
is a structural embedding. But the one-to-one morphism
$F^{-1}(i) \colon F^{-1}(\X) \to \QT_1$ is \emph{not} a structural embedding:
its image $\{a\}$ does not form an induced
substructure of~$\QT_1$, as it is not closed under~$f$.
(The morphism $F^{-1}(i)$ is an embedding in~$\CX_1$ in the
concrete category-theoretic sense; see Ad\'amek, Herrlich and
Strecker~\cite[Definition~8.6]{AHS:Joy}.)

In the strong dual category $\CX_1$, the structural embeddings
correspond exactly to surjections in the quasivariety $\CA = \ISP\QB$.
The dual category $\CX_0$ has more structural embeddings.
\end{remark}

Comparing Lemmas~\ref{lem:se-functor} and~\ref{lem:transfer2},
we see that it is the non-categorical nature of structural embeddings that
allows a finite algebra to have truly different full dualities.


\end{document}